\documentclass[reqno,11pt]{amsart}
\usepackage{amsthm,amsfonts,amssymb,euscript,  
mathrsfs,graphics,color,amsmath,latexsym,marginnote}  
\usepackage{cite}       
\usepackage{dsfont}            
\usepackage[english]{babel}  
\usepackage{mathtools}  
\usepackage{times}

\numberwithin{equation}{section}

\usepackage{datetime}
\usepackage{lipsum}
\usepackage{enumerate}
\usepackage{tikz}
\usetikzlibrary{matrix}
\usepackage{hyperref}
\RequirePackage{amscd}
\RequirePackage{epic}
\RequirePackage[all]{xy}
\RequirePackage{url}
\RequirePackage[shortlabels]{enumitem}
\usepackage{empheq}
\usepackage{epsfig}
\usepackage{lineno} % per numerare le righe 
 \usepackage{perpage}
 \usepackage[english]{babel}

\usepackage[utf8]{inputenc}
\usepackage{cite}
\RequirePackage{amscd}
\RequirePackage{epic}
\RequirePackage{eepic}
\RequirePackage[all]{xy}
\RequirePackage{url}
\RequirePackage[shortlabels]{enumitem}
\usepackage{empheq}
\usepackage{nomencl}
\usepackage{epsfig}
\usepackage{graphicx}
\theoremstyle{plain}

\newtheorem{theorem}{Theorem}[section]
\newtheorem{proposition}[theorem]{Proposition}
\newtheorem{lemma}[theorem]{Lemma}
\newtheorem{corollary}[theorem]{Corollary}

\newtheorem{remark}[theorem]{Remark}

\newtheorem{hypo}[theorem]{Hypothesis}
\newtheorem{definition}[theorem]{Definition}

\renewcommand{\Re}{\mathrm{Re}\,}

\providecommand{\vect}[2]{{\bigl[\begin{smallmatrix}#1\\#2\end{smallmatrix}\bigr]}}   
\providecommand{\sm}[4]{{\bigl[\begin{smallmatrix}#1&#2\\#3&#4\end{smallmatrix}\bigr]}}

\newcommand{\x}{\xi}
\newcommand{\pa}{\partial}
    \newcommand{\set}[1]{{\left\{#1\right\}}}
\newcommand{\norm}[1]{{\left |#1\right |}}
\RequirePackage{epic}
\RequirePackage{eepic}
\RequirePackage[all]{xy}
\RequirePackage{url}
\RequirePackage[shortlabels]{enumitem}
\usepackage{empheq}
\usepackage{nomencl}
\usepackage{epsfig}
\usepackage{graphicx}

\newcommand{\T}{\mathbb{T}}
\newcommand{\Z}{\mathbb{Z}}
\newcommand{\R}{\mathbb{R}}
\newcommand{\C}{\mathbb{C}}

\renewcommand{\Re}{\operatorname{Re}}

\newcommand{\co}[1]{\textit{#1}}%corsivo
\newcommand{\gr}[1]{\textbf{#1}}%grassetto

\usepackage{amsthm}

\newcommand{\s}{{\sigma}}
\newcommand{\ii}{{\rm i}}

\def\norma#1{\|#1\|}

\newcommand{\ov}{\overline}

\newcommand{\N}{{\mathbb N}}

\newcommand{\cF}{{\mathcal F}}

\newcommand{\cQ}{{\mathcal Q}}

\newcommand{\tC}{{\mathtt{C}}}

\usepackage{bm}

\newcommand{\0}{{(0)}}
\newcommand{\e}{{\varepsilon}}

\newcommand{\uno}{{\mathbb I}}

\newcommand{\nnorm}[1]{{\left\vert\kern-0.25ex\left\vert\kern-0.25ex\left\vert #1 
    \right\vert\kern-0.25ex\right\vert\kern-0.25ex\right\vert}}
\usepackage{framed,enumitem} 

\newcommand{\jjap}[1]{\lfloor #1 \rfloor}

\definecolor{aqua}{RGB}{10,150,200}

\makeatletter
\def\l@subsection{\@tocline{2}{0pt}{2.5pc}{5pc}{}}
\def\l@subsubsection{\@tocline{3}{0pt}{4.5pc}{5pc}{}}
\begin{document}

\title[On the lifespan and control of high norms for NLS on tori]{On the lifespan of solutions 
and control of High Sobolev norms
for the completely resonant NLS on tori}

\date{}

\author{Roberto Feola}
\address{\scriptsize{Dipartimento di Matematica e Fisica, Universit\`a degli Studi RomaTre, 
Largo San Leonardo Murialdo 1, 00144}}
\email{roberto.feola@uniroma3.it}

\author{Jessica Elisa Massetti}
\address{\scriptsize{Dipartimento di Matematica e Fisica, Universit\`a degli Studi RomaTre, 
Largo San Leonardo Murialdo 1, 00144}}
\email{jessicaelisa.massetti@uniroma3.it}

%\thanks{ {\em Acknowledgements.} ....}  
 

 %35Q55 NLS equations (nonlinear Schr ̈odinger equations) 
 %35S50 Paradifferential operators as generalizations of partial differential operators in context of PDEs
 %35A01 Existence problems for PDEs: global existence, local existence, non-existence

\begin{abstract}    
We consider a completely resonant nonlinear Schr\"odinger equation 
on the $d$-dimensional torus, for any $d\geq 1$, 
with  polynomial nonlinearity of any degree $2p+1$, $p\geq1$, which is 
 gauge and translation invariant. We study the behaviour of 
 \emph{high} Sobolev $H^{s}$-norms of solutions, $s\geq s_1+1 > d/2 + 2$, 
 whose initial datum  $u_0\in H^{s}$ satisfies an appropriate smallness condition 
 on its \emph{low} $H^{s_1}$ and $L^2$-norms respectively.   
 We prove a polynomial  upper bound on the possible growth of the Sobolev norm
 $H^{s}$ over finite but long  time scale that is exponential in the regularity parameter $s_1$. 
 As a byproduct we get stability of the low  $H^{s_1}$-norm over such time interval.
 A key ingredient in the proof is  the introduction of a suitable ``modified energy" 
 that provides an a priori upper bound on the growth. 
 This is obtained by combining para-differential techniques 
 and suitable tame estimates.   
\end{abstract}

\maketitle 
\setcounter{tocdepth}{2}
\tableofcontents

\section{Introduction and main results}
 
Let $\T^d :=(\mathbb{R}/2\pi \mathbb{Z})^{d}$, $d\ge 1$. 
In this paper we consider  the defocusing/focusing 
Schr\"odinger equation on the $d$-dimensional torus
\begin{equation}\label{QLNLS}
{\rm i} u_{t}-\Delta u\pm |u|^{2p}u=0 \,,\;\;\; u=u(t,x) \quad \quad x\in \mathbb{T}^{d}
 \end{equation}
 where  $p\geq1$, $t\in\mathbb{R}$,
 and $\Delta$ denotes the Laplacian operator acting 
 on the Fourier basis of periodic functions by linearity as
 \[
 \Delta e^{{\rm i} j \cdot x}=-|j|^{2}e^{{\rm i} j\cdot x}\,, 
 \qquad 
 |j| := \|j\|_{\ell_2}\,,\quad \forall\, j\in \mathbb{Z}^{d}\,.
 \]
For any $s\in\R$,
 we consider  the standard Sobolev space $H^{s}(\mathbb{T}^{d};\C)$
endowed with the norm $$\|\cdot\|_{H^{s}}:=\|\cdot\|_{L^2}+\|(\sqrt{-\Delta})^s\cdot\|_{L^{2}},$$
where $(\sqrt{-\Delta})^s$ is  defined spectrally as 
$(\sqrt{-\Delta})^s e^{{\rm i} j\cdot x} = |j|^s e^{{\rm i} j\cdot x}$ for any $j\in\Z^d$.

\noindent
Given  $u_0\in H^{s}(\mathbb{T}^{d};\C)$
let $u(t)$ be the solution of the Cauchy problem  associated to \eqref{QLNLS}
 with initial condition $u(0)=u_0$.
 Our aim  is to discuss the life span, and possibly the behaviour for long time, of $u(t)$. 
 The Cauchy problem  has been widely investigated in the literature from many
 different points of view and has been handled 
with various techniques, depending 
 on the manifold  where the space variable $x$ lives, its dimension,
 and the kind of nonlinearity involved.
 
 \noindent
The general local well-posedness theory in Sobolev spaces 
(i.e. the study of existence of solutions and regularity of the flow, at least on a small interval of time)
is quite well established, see for instance \cite{LinaresPonce}
%\cite{cazenaveBook, LinaresPonce} 
and reference therein for a wide discussion 
on this topic. 
Moreover, regarding Schr\"odinger type equations,  
local existence in time is known also in the case of 
very general \emph{quasi-linear} nonlinearities. 
We mention for instance \cite{KVP, MMT3}
% \cite{colin, KVP, MMT3} 
and reference therein for the equation posed on
the  \emph{Euclidean space} ($x\in\mathbb{R}^{d}$, $d\geq1$), and \cite{FI2022} 
%\cite{FI2018, FI2022} 
for the local well-posedness on $\mathbb{T}^{d}$, by means of different techniques.

 The description of the behaviour of local solutions for \emph{long} time is more involved 
 and requires additional arguments with respect to the local theory.
 On the Euclidean space 
 it is well know that dispersive effects, such as decay  in time of the norms of linear solutions,
 could help to control Sobolev norms on large time scales.
 Obviously the use of such decay properties strongly depends both on the dimension 
 and on the specific nonlinearity at stake, however there are striking results 
 proving even global well-posedness (i.e. a control on the solution for any time). 
 See for instance the seminal works \cite{BreGal80}, and \cite{GVelo, CollianderEco}. 

On the contrary, on compact manifolds the decay of solutions is no more 
a key tool for studying the well-posedness problem. 
Of course, conservation laws may help in getting global solutions
or long time existence, but they rely on very specific \textit{integrable} models. 
See for instance \cite{KaP, GreKapBook} and references therein. 
Concerning global well-posedness for non-integrable models of NLS, 
as notable examples 
we refer to Bourgain \cite{Bou1999} for $\mathbb{T}^{d}$ up to dimension $3$ 
and to Burq-G\'erard-Tzvetkov \cite{BGT04, BGT05} for  generic compact 
manifolds of dimension $2$ and $3$; in the $3$-dimensional cases 
of these works only polynomial nonlinearity of low degree are considered. 
We send the interested reader to the more recent results 
\cite{HerrTatTzve2011, IonescuPaus2012} on the quintic NLS for small/big initial data, 
%for small/big initial data 
%\cite{HerrTatTzve2011, IonescuPaus2012} in the quintic NLS, 
and \cite{PlanchonTzvetVisc2017} for a different approach based on \textit{modified energy} techniques. 
 
It is worthwhile to mention that on compact manifolds 
the dynamics is governed by the interaction 
 between linear and nonlinear frequencies of oscillations and possible resonant effects. 
 A dynamical system approach, at least in the case of small-initial data, 
 then revealed to be an effective viewpoint to handle the study of the evolution of solutions. 
Let us be more precise. Many of the PDEs under study, 
after possibly appropriate variable's change, can be formulated as 
 $ \ii \dot u = L u + N(u)$
  where $u$ lives in some Banach space, $L$ is an appropriate self-adjoint operator 
  depending on the problem, and $N$ is a nonlinear term $\sim O(|u|^{q+1})$, $q\ge 1$.
  In the case of initial data of size $\epsilon\ll1$, where the nonlinearity 
  is viewed as a perturbation of the linear flow,  classical \emph{local in time} theory guarantees  
  that the solutions have actually a \emph{linear} behaviour (hence stable), 
  in the sense that they are governed by  oscillatory motions with \emph{linear} frequencies 
  given by the eigenvalues of $L$, for times  of order $t\sim O(\epsilon^{-q})$. 
  We shall refer to this threshold of time as (trivial) \emph{nonlinear time scale}.
  
  \noindent
 For longer time scales the effect of the nonlinearity becomes non-trivial
and one expects that the dynamics is led by the resonant interactions 
 of the linear frequencies of oscillations.
 In order to investigate this phenomenon, 
 Birkhoff normal form's approach has been extensively used as an effective tool
 for proving long time (or almost global) existence and stability of solutions.
We quote the seminal work of
%Bourgain \cite{Bourgain:1996}, 
Bambusi \cite{Bambusi:2003}  and Bambusi-Gr\'ebert \cite{BG}
where the authors
consider (among other models) the Klein-Gordon equation on the circle.
They prove almost global existence in the sense that, 
for  any $N\geq1$ and any initial datum in $H^{s}(\mathbb{T};\mathbb{C})$
of size $\epsilon\ll1$ with $s\gg1$ large enough, 
the solution exists and its $H^{s}$-Sobolev norm
remains small over a time interval of size  $O(\epsilon^{-N})$.
Similar results have been obtained
for semi-linear equations on different manifolds, see for instance \cite{BDGS, DS0},
while, for quasilinear equations, we refer to
\cite{Delort-2015, Berti-Delort, %FI,
BFP2022, BFF2021}.

Note that all the above results require very strong non-resonant 
conditions on the frequencies of oscillations so that an additional linear term 
(e.g. a multiplicative or convolution potential, a mass...) 
must be introduced in the equation. 
This term provides additional free parameters
that enable one to modulate the frequencies 
and keeping the non-resonance requirement. 
It is well-known that these conditions are quite difficult to met in dimension greater or equal than $2$, 
even with an external potential at disposal. See for instance 
\cite{Delort:2009vn, FangZhang10, Imekraz2016, BFGI2021, FGI20, FIM2022}.

In this kind of study, observe that Sobolev regularity in data usually corresponds 
to a lifespan of solutions which is polynomial in terms of the size of the initial datum. 
More recently, some results about exponential stability-times 
have been achieved, but they all rely on models in one space dimension 
and in presence of an external potential as source of parameters. 
Of course, it is natural to expect exponential-type times whenever the
initial data are very regular like analytic or Gevrey for example, 
see for instance \cite{Faou-Grebert:2013, CongMiShi22}.
 
On the other hand, in the case of Sobolev initial data, a refined 
balance between regularity and size has been exploited 
in \cite{BMP:CMP, BMP:linceiStab} 
to get exponential-type results for NLS with convolution potential. 
See also \cite{FeoMass1} for a degenerate case, the Beam equation, 
where only one parameter (the mass) is at disposal for the frequency modulation. 
This connection between time and regularity is a common feature 
also in the finite dimensional context, see for instance 
\cite{Bounemoura_Fayad_Niederman_2020_2, BMM:CMP} 
and reference therein.

\vspace{0.6em}
\textbf{Results.} In this work we study the problem of long time 
existence of solutions and possible stability in the context of 
a \textit{completely resonant} model 
(i.e. the linearized model admits an infinite dimensional subspace of periodic solutions) 
in high space dimension, 
where the linear frequencies of oscillations 
involve infinitely many resonant interactions.
This is the case of equation \eqref{QLNLS}, 
where the linear frequencies of oscillations 
are given by $|j|^{2}$ for $j\in\Z^{d}$, i.e. the spectrum of $-\Delta$ on $\T^{d}$.
Since all these frequencies are integer, 
one trivially sees that a diophantine equation like 
\[
|j_{1}|^{2} \pm |j_{2}|^{2}\pm \cdots \pm |j_{n}|^{2} = 0\,, 
\qquad 
j_{i}\in\Z^{d}\,, \;\;\; i = 1,\ldots, n\,, 
\]
 which represents a resonant $n$-wave interaction with $n\geq 2$, 
 admits infinitely many solutions.  
 A Birkhoff Normal Form approach would \emph{reduce} the system 
 to a resonant one plus a non-resonant one of smaller size. 
 Then, the leading dynamics, involving such a large resonant subset,  
 may give birth to instability and complicated phenomena, see for example \cite{CKSTT}.
Note that in BNF context \textit{long time} existence is a byproduct of the stability of initial data. 
Of course, it is quite hard to implement this idea in completely resonant systems.

 In this resonant context, we shall introduce a different paradigm based on a tameness property 
 of a Sobolev-type norm $\norm{\cdot}_{s}$, 
 available on the scale  $(H^{s}, \norm{\cdot}_{s})$ with $s\ge s_{1} > d/2$,  
 which is equivalent to the classical norm $\norma{\cdot}_{H^s}$ 
 (see Definition \ref{def:equinorm} and Lemma \ref{lemmaequivalenza}). 
 More precisely, inspired by the observation in \cite[Remark $1.4$]{BMP:CMP} 
 (see also \cite{CongMiShi22}), 
 we take advantage of the rescaling from high to low Sobolev indices 
 in order to gain a sharper control on the size of the nonlinearity. 
 This produce stability of the norm $\norma{\cdot}_{H^{s}}$ 
 over a time scale which is exponentially long in the regularity parameter $s_{1}$. 
 
 \noindent
 This is the content of Theorem \ref{thm:main1} stated below. 
 By combining that with appropriate para-differential calculus, we construct a suitable 
 modified energy that enables us to have \textit{a priori} upper bounds 
 on the possible growth of the \textit{high} Sobolev norm $\norma{\cdot}_{H^{s}}$, 
 for finite but long times that depend only on the 
 \textit{low} norm $\norma{\cdot}_{{H^{s_{1}}}}$. 
 The result holds for all $H^{s}$-initial data, under appropriate 
 smallness condition  both on the low and $L^{2}$ norms. 
 
The model \eqref{QLNLS}, prototype  of NLS equations with polynomial
nonlinearity of \emph{odd} degree $|u|^{2p}u$, 
posed on the \emph{straight} torus $\mathbb{T}^{d}$ is widely studied. 
We stress the fact that  our approach applies 
 to \gr{any} space dimension and for \gr{any} degree $p$ of the nonlinearity.
Moreover we point out that our method applies both in the \emph{defocusing} and 
 \emph{focusing}  case (sign $+$ and $-$ respectively in \eqref{QLNLS}).
 Without being redundant we choose the sign $+$ in our computations.
We believe that a generalization to \emph{generic} tori 
 and more general nonlinearities is only a technical issue. 
 For the sake of clarity we choose to work on this simpler (non trivial) model
to enlighten the strength of the method itself. 

The core of the paper is the following 

\begin{theorem}{\bf (Long time estimates on high Sobolev norms).}\label{thm:main3}
Fix $0<\e\ll1$, $p\geq1$, $s_0>d/2$,  consider $s_1\geq s_0+2$.
There exist  absolute constants $\mathtt{C}\geq \mathtt{M}>0$ such that
for any $s\geq s_1+1 $ and any 
$u_0\in H^{s}(\T^{d};\C)$ such that
\begin{equation}\label{piccolezza dati tris}
\|u_0\|_{H^{s_1}}\leq \e\,,
\qquad 
(2\mathtt{M})^{s_1}\|u_0\|_{L^{2}}\leq \e\,,
\qquad 
\|u_0\|_{H^{s}}<+\infty\,,
\end{equation}
the following holds.
There exist a time $T=T(u_0,\mathtt{M},s_1)>0$
and a unique solution $u(t)$ of \eqref{QLNLS} with initial condition $u(0)=u_0$ 
such that
\begin{equation}\label{patata1}
\begin{aligned}
&u\in C^{0}\big([0,T]; H^{s}(\T^{d};\C)\big)\cap 
C^{1}\big([0,T]; H^{s-2}(\T^{d};\C)\big)\,,
\\& T\geq T_{good}:=\frac{1}{\e^{2p}}
\frac{1}{2^{4(2p+2)}}\frac{2^{2p(s_1-s_0)}}{\mathtt{M}^{2ps_0}}\,.
\end{aligned}
\end{equation} 
Moreover one has
\begin{equation}\label{stimaIncredibleBis}
\|u(t)\|_{H^{s}}
\leq 
 3 \mathtt{C}^{ps(s-s_1)}
 \Big(
 \|u_0\|_{H^s} + (2\mathtt{M})^{s}\|u_0\|_{L^{2}}
 \Big)
 \Big[
 1+\Big(
 \mathtt{M}^{2p(s-s_1)}\frac{2^{4p}\mathtt{M}^{2p s_0}}{2^{2p(s_1-s_0)}}\e^{2p}\, t
 \Big)^{s-s_1}
 \Big]\,,
\end{equation}
for any $t\in[0,T]$ and,
as a consequence, the bound
\begin{equation}\label{stimaIncredible}
\sup_{t\in[0,\widehat{T}]}\|u(t)\|_{H^{s}}
\leq 
3\mathtt{C}^{ps(s-s_1)}
\Big(
\|u(0)\|_{H^s}+ (2\mathtt{M})^{s}\|u_0\|_{L^{2}}
\Big)
(1+
\mathtt{M}^{2p(s-s_1)^2}) \,,
\qquad
\forall\, \widehat{T}\leq T_{good}\,.
\end{equation}
\end{theorem}

Some comments are in order.

\begin{enumerate}
\item First of all we remark that the existence result is obtained for Sobolev data in $H^{s}$
satisfying   a smallness condition \emph{only} in a low norm $s_1<s$. 
In particular the strongest requirement regards the 
$L^{2}$-norm which should be exponentially small in $s_1$.
This line of thoughts is similar to the one used in \cite{FM2022}
where the control on high norms and the life span depends only on the \emph{size} of the low 
$\|u_0\|_{H^{s_1}}$.
However here 
the situation is quite different and more favourable,  
since the lower bound on the time of existence is exponentially large w.r.t. $s_1$, 
so that one can make the lifespan as long as desired just by taking the regularity $s_1$, 
and as a consequence $s$, large enough at the price of asking for 
a very small $L^{2}$-norm of the initial condition.
For further details on this we refer to the comments below Theorem \ref{thm:main1}.

\item The upper bound \eqref{stimaIncredibleBis} on the \emph{high} norm $\|\cdot\|_{H^{s}}$
is polynomial in the time $t$, with degree depending on $s-s_1$. 
We remark that polynomial growth in $t$ is not trivial, since a straightforward application of 
a standard Gr\"onwall-type inequality would lead only to an exponential bound in $t$.
Estimate \eqref{stimaIncredibleBis} guarantees  the bound \eqref{stimaIncredible} over 
the  time interval $[0,T_{good}]$.

\item  The drawback of working in any dimension 
and considering any $p$ in the nonlinearity is that
we  are forced to work with sufficiently regular spaces where algebra property holds.
 The threshold $s>d/2+3$ is required to obtain 
 the basic nonlinear estimates provided in Lemma \ref{lem:confort2} and 
 to perform some \emph{paradifferential changes of coordinates}
 in subsection \ref{sec:improved} that we need 
 in order to construct a suitable modified energy for the solutions.

\item 
We  believe that our techniques,  combined with some of the ideas in \cite{FM2022},
could be extended to the case \emph{generic} tori $\mathtt{T}^{d}_{\Gamma}:=\mathbb{R}^{d}/\Gamma$
where $\Gamma$ is an arbitrary periodicity lattice.
We remark also that  the nonlinearity $|u|^{2p}u$ is 
\emph{gauge} and \emph{momentum} preserving,
even if they are not strictly necessary for our aims.
More general polynomial (or analytic) nonlinearities could be considered. 
We choose this model in order to simplify notation and technicalities. 

\item We did not attempt to optimise the constant $\mathtt{M}$. 
In \cite{BMP:CMP} and \cite{CongMiShi22} $\mathtt{M}$ is fixed equal $1$. 
This choice is sufficient to prove bounds like \eqref{lem:confort2} in \cite{CongMiShi22} 
but the choice of \emph{cubic} nonlinearity $|u|^{2}u$ 
for a NLS equation seems a necessary requirement.
\end{enumerate}

As explained above, the study of high norms $\|\cdot\|_{H^{s}}$ in Theorem \ref{thm:main3}
is based on the control over a certain time scale of the low norms $\|\cdot\|_{H^{s_1}}$. 
In particular  one needs \emph{stability} of such low norms.  
This point, in a different situation emerges also in \cite{FM2022}. 
In fact, in the cited paper the stability of low norms is obtained,
along a time scale quadratic in $\e$, through Birkhoff normal form techniques. 
In the present paper  the idea is rather different, 
moreover we are able to obtain a control over a time scale that is exponential 
in the regularity parameter $s_1$. 
This fact implies even an exponentially 
long stability, w.r.t. the size $\e$, for sufficiently regular initial data. 
This is the content of the following result. 

%\begin{theorem}{\bf (Exponential type stability).}\label{thm:main1}
\begin{theorem}\label{thm:main1}
Fix $0<\e \ll1$, $s_1\geq s_0>d/2$ and 
$p\geq 1$.
There exists an absolute constant $\mathtt{M}>0$ such that
for any 
$u_0\in H^{s_1}(\T^{d};\C)$ such that
\begin{equation}\label{piccolezza dati}
\|u_0\|_{H^{s_1}}\leq\e\,,
\qquad 
(2\mathtt{M})^{s_1}\|u_0\|_{L^{2}}\leq \e\,,
\end{equation}
the following holds. 
There exist a time $T:=T(u_0,\mathtt{M},s_1)>0$ 
and a unique solution $u(t)$ of \eqref{QLNLS} with $u(0)=u_0$ and 
\begin{equation}\label{storm}
\begin{aligned}
&u\in C^{0}\big([0,T]; H^{s_1}(\T^{d};\C)\big)\cap 
C^{1}\big([0,T]; H^{s_1-2}(\T^{d};\C)\big)\,,
\\& T\geq T_{good}:=\frac{1}{\e^{2p}}
\frac{1}{2^{4(2p+2)}}\frac{2^{2p(s_1-s_0)}}{\mathtt{M}^{2ps_0}}\,,
\end{aligned}
\end{equation} 
satisfying the bound
\begin{equation}\label{stimaccia}
\sup_{t\in[0,T]}\|u(t)\|_{H^{s_1}}
\leq 
6\Big(\|u_0\|_{H^{s_1}}+(2\mathtt{M})^{s_1}\|u_0\|_{L^{2}}\Big)\,.
\end{equation}
If in addition for $s\geq s_1$ one also has  $u_0\in H^{s}(\T^d;\C)$
then the solution $u(t)$ evolving form $u_0$
satisfies \eqref{patata1} and the bound
\begin{equation}\label{stimaccia2bis}
\|u(t)\|_{H^{s}}
\leq 
3 \Big(
\|u_0\|_{H^{s}} +(2\mathtt{M})^{s}\|u_0\|_{L^2}
\Big)
\exp\big(
\mathtt{M}^{2p(s-s_1)} 
\frac{2^{4p+1}\mathtt{M}^{2ps_0}}{2^{2p(s_1-s_0)}}\e^{2p}t\big)\,,
\qquad 
\forall\; t\in[0,T]\,.
\end{equation}
As a consequence one has
\begin{equation}\label{stimaccia2}
\sup_{t\in[0,\widetilde{T}]}\|u(t)\|_{H^{s}}
\leq 
3 \Big(\|u_0\|_{H^{s}} +(2\mathtt{M})^{s}\|u_0\|_{L^{2}}\Big)
\exp\big\{\mathtt{M}^{2p(s-s_1)} \big\}\,,
\qquad 
\widetilde{T}\leq T_{good}\,.
\end{equation}
\end{theorem}

The control of the norm $\|\cdot\|_{H^{s_1}}$ over the time interval in \eqref{storm}
is similar to the one given in \cite{CongMiShi22}.
It is essentially based on some  scaling/tame type inequalities
fulfilled by a suitable equivalent Sobolev norm, 
which we introduce in Definition \ref{def:equinorm}.
We refer the reader to the crucial estimates given by 
Lemmata  \ref{interopolo} and \ref{lem:tametame}.
The basic idea behind the introduction of such kind of norm is 
to weight differently (in terms of the regularity parameter $s$)
the low frequencies of a function  w.r.t. the high ones.
This allows us to describe, in an efficient way, the fact that we essentially work 
on Sobolev solutions of \eqref{QLNLS} whose energy is {\bf not} concentrated on low modes.
Actually condition \eqref{piccolezza dati} entails exactly this requirement, 
since the second bound is satisfies only if the mass of 
the function $u_0$ is not supported on low modes. 
We refer the interested reader to Remark $1.4$ in \cite{BMP:CMP} for further comments on this.
The lower bound \eqref{storm} on the time of existence and stability reflects this fact. 
In other words if one starts with very few energy on the low modes, 
then the Sobolev norms  $\|\cdot\|_{H^{s_1}}$
are stable for  very long time in terms of $s_1$.
We remark that, by requiring that $\|u_0\|_{H^{s_1}}\leq \e$ only,
then one could only hope to get a time of stability $T\sim \e^{-2p}$, 
which is the natural nonlinear time scale.

%\smallskip
\vspace{0.6em}
On the other hand Theorem \ref{thm:main1}
provides an extra factor $2^{2p(s_1-s_0)}$ in the estimate of the lifespan.
This suggest  in particular that one can balance the \emph{regularity} of initial data with 
their  $L^2$-\emph{size} to obtain any desired long time stability of solutions.
Of course the price  to pay is asking for a smaller $L^2$-norm of the initial data.
For instance, the possible choices
\begin{equation*}
s_1:=\frac{1}{\e \ln 2} \,,
\quad \quad \text{or} \quad \quad 
s_1:=\frac{ e^{1/\e}}{\ln2}\,,
\end{equation*} 
lead to \emph{exponentially} or \emph{super-exponentially} 
long times of stability respectively, 
i.e. \eqref{stimaccia} holds true over a time interval $[0,T]$ with 
\[
T\gtrsim_{s_0,p}\e^{-2p} \exp\Big\{\frac{2p}{\e}\Big\}\,, 
\quad \quad \text{or} \quad \quad 
T\gtrsim_{s_0,p}\e^{-2p} \exp\Big\{2p\exp\{\frac{1}{\e}\}\Big\}\,.
\]
%while 
%\[
%s_1\sim e^{1/\e}\qquad \Rightarrow
%\qquad T\gtrsim_{s_0,p}\e^{-2p} \exp\exp\{\frac{2p}{\e}\}\,.
%\]

\smallskip

We mention, \emph{en passant}, that besides the control of the $H^{s_1}$-norm, 
Theorem \ref{thm:main1} gives a control of the $H^{s}$-norm for $s\geq s_1$ 
over the same time interval given in \eqref{storm}. 
However such control is a consequence of a standard Gr\"onwall-type Lemma,
that provides a rough exponential bound (from above) in time.

\noindent
The bound \eqref{stimaccia2}
is deeply and significantly improved in Theorem \ref{thm:main3} (see \eqref{stimaIncredible}).

In Theorems \ref{thm:main3}-\ref{thm:main1}
we exploit deeply the \emph{stability} of the \emph{low} norm $\|\cdot\|_{H^{s_1}}$ 
of the solution and scaling properties of the equivalent norms $|\cdot |_s$. 
One may ask whether is possible to get stability for longer times 
by combining this point of view with some steps of Birkhoff normal form.  
Of course this idea is possibly successful only on 
strongly \emph{non-resonant} equations (i.e. in presence of external potential). 
This may be investigated in a further work.

\vspace{0.6em}
\noindent
{\bf Structure of the paper and strategy.}
In Section \ref{sec:funcsett}
we introduce the function spaces on which we work. 
As we mentioned before, we shall work with a  ``modified" Sobolev norm $|\cdot|_{{s_1}}$, 
equivalent to the standard one, 
whose definition and equivalence relation are given by 
Definition \ref{def:equinorm} and Lemma \ref{lemmaequivalenza}.
We also discuss \emph{para-products} and their properties, 
such as their action on Sobolev spaces and compositions 
(see subsections \ref{suppara}-\ref{subparacompo}).

In Section \ref{sec:paraNLS} we shall rewrite equivalently
equation \eqref{QLNLS} in a para-differential form
and we provide the first basic nonlinear \emph{tame} estimates
(see Lemma \ref{lem:confort2}).
Such estimates are used in Subsection \ref{sec:basicbasic}
for proving a first \emph{energy inequality} of the form
\[
|u(t)|^{2}_{{s_1}}\leq |u(0)|^{2}_{{s_1}}
+C\int_{0}^{t}|u(\s)|_{s_0}^{2p}|u(\s)|_{s_1}^{2}d\s\,,
\]
for some constant $C>0$, 
which provides an \emph{a priori} control on the
Sobolev norm $|\cdot|_{{s_1}}$, $s_1\geq s_0>d/2$ of the solution.
Thanks to the tame/scaling properties of the norm
the energy estimate above can be improved to
\[
|u(t)|^2_{{s_1}}\leq |u(0)|^2_{{s_1}}
+\widetilde{C}2^{-2p(s_1-s_0)}\int_{0}^{t}|u(\s)|_{s_1}^{2p+2}d\s\,.
\]
This is the content of Theorem \ref{thm:energyBasic}.
The extra factor $2^{-2p(s_1-s_0)} $
allows to prove the stability of the solutions over a time scale as in \eqref{storm}.
Combining the energy estimates above with a Gr\"onwall-like lemma,
one gets \emph{a priori} bounds on the high norms.
This is done in Theorem \ref{thm:energyHighhigh}.

\vspace{0.6em}
\noindent
The proof of the main result Theorem \ref{thm:main3}
is more subtle and requires an additional, novel approach.
The key idea, implemented in subsection \ref{sec:improved},
is to construct a \emph{modified} energy $E_{s}(t)$ that enjoys the following relations: 
one has
\begin{equation}\label{energiamigliore}
E_{s}(t) \sim |u(t)|^2_{s}\,, 
\qquad \quad 
E_{s}(t)\lesssim_{s} 
E_{s}(0)+C\int_{0}^{t}f(\s)E_{s}^{\alpha}(\s)d\s\,,
\end{equation}
where $f(t)$ depends only on the \emph{low} norm
$|u(t)|_{s_1}$ and $\alpha<1$ is strictly less than one.
This is the content of Theorem \ref{thm:energy}.

\noindent
The main step in getting the above relations is  the introduction 
of an auxiliary variable $w$ which solves 
an equation equivalent to the NLS \eqref{QLNLS}.
Then we set $E_{s}(t):=|w|^2_{s}$. 
Corollary \ref{equinono} guarantees the equivalence in \eqref{energiamigliore}. 
Then we construct a para-differential change of variables, 
see Proposition \ref{prop:block}, 
that allows to show that
\[
|w(t)|_{s}^{2}\lesssim_{s}
|w(0)|_{s}^{2}
+C\int_{0}^{t}f(\s)|w(\s)|_{s-1}|w(\s)|_{s}d\s\,.
\]
Since the norm $|\cdot|_{s}$ enjoys the classical \emph{interpolation estimates}
in Sobolev spaces we get
\[
|w(\s)|_{s-1}|w(\s)|_{s}\lesssim_s 
|w(\s)|_{s_1}^{1-\lambda}|w(\s)|^{\lambda}_{s}\,,\qquad 
\lambda=1-\frac{1}{(s-s_1)}\,,
\]
by the assumption $s\geq s_1+1$,
which implies the estimate in \eqref{energiamigliore}, by choosing 
$\alpha = 1-\tfrac{1}{2(s-s_1)} $.
Combining the improved energy estimates with an improved 
Gr\"onwall lemma (Lemma \ref{drago}-(ii) in the appendix)
we are able to get a sharper control on the high norms $|u(t)|_{s}$. 
In particular we show that these norms \emph{may} 
grow to infinity at most at a polynomial rate in $t$, 
this proving Theorem \ref{thm:main3}.

\section{Functional setting}\label{sec:funcsett}
As usual, we identify the Sobolev space of $2\pi$-periodic functions 
$\mathbb{\R}^{d}\ni x \mapsto u(x)\in \mathbb{C}$ with 
$H^{s}(\mathbb{T}^d;\mathbb{C})$, and develop any function $u: \T^d\to \C$ 
in its Fourier series as\footnote{We also use the notation
$u_n^+ := u_n := \widehat{u}(n)$ and
$u_n^- := \ov{u_n}  := \ov{\widehat{u}(n)}$.}
\begin{equation*}
u(x) = \frac{1}{(2\pi)^{d/2}}
\sum_{n \in \Z^{d} } \widehat{u}(n)e^{\ii n\cdot x } \, , \qquad 
\widehat{u}(n) := \frac{1}{(2\pi)^{{d}/{2}}} \int_{\mathbb{T}^{d}} u(x) e^{-\ii n \cdot x } \, dx \, .
\end{equation*}
We endow $H^{s}(\mathbb{T}^{d};\mathbb{C})$ with the norm 
\begin{equation*}
\|u\|_{H^{s}}:=\|u\|_{L^{2}}+\|(\sqrt{-\Delta})^{s}u\|_{L^{2}}\,.
\end{equation*}

\noindent 
Considering the Fourier expansion of periodic functions,
it is common habit to work on the equivalent Sobolev norm
\begin{equation*}
|u|^{2}_{H^{s}}=(\langle D\rangle^{s}u,\langle D\rangle^{s} u)_{L^{2}}=
\sum_{j\in \mathbb{Z}^{d}}\langle j\rangle^{2s}|u_{j}|^{2} 
\qquad
\langle D\rangle e^{\ii j\cdot x} :=\langle j\rangle  e^{\ii j\cdot x}\,,\;\;\; 
\forall \, j\in \mathbb{Z}^{d}\,,
\end{equation*}
where $(\cdot,\cdot)_{L^{2}}$ is the standard complex $L^{2}$-scalar product
 \begin{equation*}
 (u,v)_{L^{2}}:=\int_{\mathbb{T}^{d}}u\bar{v}dx\,, 
 \qquad \forall\, u,v\in L^{2}(\mathbb{T}^{d};\mathbb{C})\,,
 \end{equation*}
 and $\langle D\rangle $ is the linear operator, defined by
 \begin{equation}\label{langlesimbolo}
 \langle D\rangle e^{\ii j\cdot x} :=\langle j\rangle  e^{\ii j\cdot x}\,,
 \qquad \langle j\rangle:=\sqrt{1+|j|^{2}}\,,\quad j\in \Z^{d}\,.
 \end{equation}
 Indeed,  for any $s\geq1$, one can easily check that\footnote{Recall that 
 $(x+y)^{q}\leq 2^{q-1}(x^{q}+y^{q})$ for all $x,y\geq0$.\label{ciaoneproprio}}
 \[
\tfrac{1}{\sqrt{2}}\Big( \|u\|_{L^{2}}+\|\pa_{x}^{s}u\|_{L^{2}}\Big)
\leq |u|_{H^{s}}\leq 2^{\frac{s-1}{2}}\Big(\|u\|_{L^{2}}+\|\pa_{x}^{s}u\|_{L^{2}}\Big)\,.
 \]

\noindent
{\bf Notation}. 
We shall use the notation $A\lesssim B$ to denote 
$A\le C B$ where $C$ is a positive constant
depending on  parameters fixed once for all, 
for instance $d$ and $s$.
 We will emphasise by writing $\lesssim_{q}$
 when the constant $C$ depends on some other parameter $q$.
 We shall write $\lessdot_{s}$
 if there exist a constant $C>0$ (independent of $s$)
 such that $A\leq C^{s}B$.  
  We shall write $\lessdot_{s,s_0}$
 if there exist a constant $C>0$ (independent of $s$ and $s_0$)
 such that $A\leq C^{\max\{s,s_0\}}B$. 
 
 \noindent
Moreover, for $r\in\R^{+}$, we shall
denote by $B_{r}(H^{s}(\mathbb{T}^{d};\mathbb{C}))$
the open ball of $H^{s}(\mathbb{T}^{d};\mathbb{C})$ 
with radius $r$ centred at the origin.

For our purposes, on the space $H^{s}(\T^{d};\C)$ 
we shall define and make use of the following equivalent norm, 
instead of the usual one.

\begin{definition}{\bf (Equivalent norm).}\label{def:equinorm}
Let $s, R\in \R$ with $R>1$. 
Let us define the operator  $\jjap{D}$ defined by 
linearity as 
\begin{align}
&\jjap{D}e^{\ii j\cdot x}=\jjap{j}^{2}e^{\ii j\cdot x}\,,
\qquad \jjap{j}:=\jjap{j}_{R}:=\max{\{R,|j|\}}\,,\quad j\in\Z^{d}\,,
\label{def:japjapModificato}
\end{align}
For any $u\in H^{s}(\T^{d};\C)$ we define
\begin{align}
|u|^{2}_{s}&:=|u|^{2}_{s,R}:=
(\jjap{D}^{s}u,\jjap{D}^{s}u)_{L^{2}}
=\sum_{j\in\Z^{d}} \jjap{j}^{2s} |\widehat{u}(j)|^{2}\,.\label{normaJapJap}
\end{align}
\end{definition}

\begin{remark}\label{rmk:facile}
Notice that $\langle j\rangle\leq 2\jjap{j}$  
and 
$\jjap{j}\geq {R}$ for any $j\in\Z^{d}$.
\end{remark}
\begin{remark}\label{scalaBan}
Notice that, for any $s_1\geq s_2\geq s_0>d/2$ 
and any $u\in H^{s_1}(\T^{d};\C)$ one has
$|u|_{s_2}\leq |u|_{s_1}$.
\end{remark}

\begin{lemma}{\bf (Equivalence of the norms).}\label{lemmaequivalenza}
Let $s,R>0$. For any $u\in H^{s}(\T^{d};\C)$, 
\begin{align}
\|u\|_{H^s}+R^{s}\|u\|_{L^{2}}
&\leq 
3 |u|_{s}
\leq  
3\big(
\|u\|_{H^s}+R^{s}\|u\|_{L^{2}}
\big)\,.\label{equiIncredibile1}
\end{align}
\end{lemma}

\begin{proof}
It is easy to check (using Remark \ref{rmk:facile}) that
$\|u\|_{H^{s}}\leq |u|_{s}$ and $\|u\|_{L^{2}}^{2}\leq R^{-s}|u|_{s}^2$. 
One has that
\begin{equation*}
\begin{aligned}
\|u\|_{H^s}+R^{s}\|u\|_{L^{2}}&=\Big(\sum_{j\in\Z^{d}}R^{2s}|\widehat{u}(j)|^{2}\Big)^{1/2}
+\Big(\sum_{j\in\Z^{d}}|\widehat{u}(j)|^{2}\Big)^{1/2}
+\Big(\sum_{j\in\Z^{d}}|j|^{2s}|\widehat{u}(j)|^{2}\Big)^{1/2}
\\&\leq 3\Big(\sum_{j\in\Z^{d}}(\max{R,|j|})^{2s}|\widehat{u}(j)|^{2}\Big)^{1/2} \,.
\end{aligned}
\end{equation*}
The we deduce the \eqref{equiIncredibile1}.
\end{proof}
We give the following Definition.
\begin{definition}{\bf (Projectors).}\label{def:projectors}
Given $N>1$, we define 
the \emph{projector} $\Pi_{N}$ as
\[
\Pi_{N}u=\frac{1}{(2\pi)^{d/2}}
\sum_{\substack{|j|\leq N}}
\widehat{u}(j)e^{\ii j\cdot x}\,,
\qquad 
u\in H^{s}(\T^{d};\C)\,.
\]
We set $\Pi_{N}^{\perp}:=\uno-\Pi_{N}$ where $\uno$ is the identity.
\end{definition}

We have the following classical result.
\begin{lemma}\label{interopolo}
Let $s>0$. $(i)$ One has
\begin{align*}
|\Pi_{N}u|_{s+\beta}&\leq \max\{R,N\}^{\beta}|u|_{s}\,,
\qquad 
\forall\, u\in H^{s}(\T^{d};\C) \,,\; \beta\geq0
\\
|\Pi_{N}^{\perp}u|_{s}&\leq \max\{R,N\}^{-\beta}|u|_{s+\beta}\,,
\qquad 
\forall\, u\in H^{s+\beta}(\T^{d};\C) \,,\; \beta\geq0
\end{align*}
$(ii)$ For $0\leq s_1\leq s\leq s_2$, $s:=\lambda s_1+(1-\lambda)s_2$, one has
\begin{equation*}%\label{albero3}
|u|_{s}\leq 2|u|_{s_1}^{\lambda}|u|_{s_2}^{1-\lambda}\,,
\qquad 
\forall\, u\in H^{s_2}(\T^{d};\C)\,.
\end{equation*}
\end{lemma}

\begin{proof}
Item $(i)$
follows straightforward using \eqref{def:japjapModificato}.
Let us check item $(ii)$. Fix $N\geq R$ 
(being $R$ the constant appearing in \eqref{def:japjapModificato}).
Then, using item $(i)$, 
we have
\[
|u|_{s}\leq |\Pi_{N}u|_{s}+|\Pi^{\perp}_{N}u|_{s}
\leq
%\stackrel{\eqref{albero1}, \eqref{albero2}}{\leq}
N^{s-s_1}|u|_{s_1}+N^{-(s_2-s)}|u|_{s_2}\,,
\]
which implies the interpolation estimate of item $(ii)$ by 
setting 
$N^{s-s_1}:=|u|_{s_1}^{\lambda-1}|u|^{1-\lambda}_{s_2}$.
\end{proof}

\begin{remark}\label{scaling property}
In the same spirit of Lemma \ref{interopolo}, 
using Remark \ref{rmk:facile}, we get,  for $s\geq s_0>d/2$,
\begin{align}
|u|_{s_0}&\leq R^{-(s-s_0)}|u|_{s}\,,\qquad u\in H^{s}(\T^d;\C)\,.
\label{confort3}
\end{align}
\end{remark}

\subsection{Para-products}\label{suppara}
We introduce an important class of operators we will use along the paper.

\noindent
Let $0<\epsilon< 1/4$, consider
a smooth function $\chi : \mathbb{R}\to[0,1]$ satisfying 
\begin{equation}\label{cutofffunct}
\chi(\x)=\left\{
\begin{aligned}
&1 \quad {\rm if} |\x|\leq 5/4 
\\
&0 \quad {\rm if} |\x|\geq 8/5 
\end{aligned}\right.
\end{equation}
and define
\begin{equation}\label{cutofffunctepsilon}
\chi_{\epsilon}(\x):=\chi(|\x|/\epsilon)\,.
\end{equation}

\noindent
- Let $s>d/2$ and $0<\epsilon< 1/4$.  To a function $a\in H^{s}(\T^{d};\C)$  we associate the
\emph{para-product}  operator $a \mapsto T^\epsilon_a$, 
that is the linear operator defined as 
\begin{equation}\label{paraprod}
h\mapsto T^\epsilon_{a}h:=\frac{1}{(2\pi)^{d}}\sum_{j\in \mathbb{Z}^{d}}e^{\ii j\cdot x}
\sum_{k\in\mathbb{Z}^{d}}
\chi_{\epsilon}\Big(\frac{|j-k|}{\langle k\rangle}\Big)
\widehat{a}(j-k)\widehat{h}(k)\,,
\end{equation}
where $\langle k\rangle $ is the usual bracket defined in \eqref{langlesimbolo}.

\smallskip
\noindent
- Let $0<\epsilon_2\leq \epsilon_1<1/4$, 
we introduce the \co{`regularizing reminder"}  $R^{\epsilon_1,\epsilon_2}_a$, defined as
\begin{equation}\label{natale}
R^{\epsilon_1,\epsilon_2}_{a}h:= 
T^{\epsilon_1}_a h - T^{\epsilon_2}_a h 
= 
\frac{1}{(2\pi)^{d}}\sum_{j\in \mathbb{Z}^{d}}e^{\ii j\cdot x}
\sum_{k\in\mathbb{Z}^{d}}
\big(\chi_{\epsilon_1}-\chi_{\epsilon_2}\big)\Big(\frac{|j-k|}{\langle k\rangle}\Big)
\widehat{a}(j-k)\widehat{h}(k)\,.
\end{equation}

\begin{remark}\label{positivity}
In the lemma below, we shall prove that the operator $T^\epsilon_a$ 
is continuous in $H^s(\T,\C)$, for any real $s$,  
while $R^{\epsilon_1,\epsilon_2}_a$ smooths. 
Hence  the operator $T^\epsilon_a$ is uniquely defined $\chi_{\epsilon}$ 
 ``up to regularizing reminders". 
\end{remark}

 \vspace{0.6em}
 \noindent
\co{Complex-conjugate and adjoint.}
%\label{strutAlg}
 Given any linear operator   $A$ on $L^{2}(\T^{d};\C)$
 we define the operator $\bar{A}$ as $\bar{A}[h]=\overline{A[\bar{h}]}$
 for any $h\in L^{2}(\T^{d};\C)$.
 Moreover, if $A:=T_{a}$ for some function $a$, 
then one can check the following:
\begin{align*}
{\bf (Complex-conjugate)}\;\;\;\;  \bar{A}[h]&:=\ov{A[\bar{h}]}\,,
\quad  \Rightarrow\quad 
\bar{A}=T_{\overline{a}}\,,
\\
{\bf (Adjoint)} \;\;\;\;  (Ah , v)_{L^{2}}&=:(h,A^{*}v)_{L^{2}}\,, 
\quad \Rightarrow \quad 
A^{*}=T_{\ov{a}}\,,
\end{align*}
where $\bar{a}$ is the complex conjugate of the function $a$.
If the symbol $a$ is real valued then the operator $T_{a}$ is then self-adjoint.

\begin{lemma}{\bf (Action on Sobolev spaces).}\label{azioneSimboo}
Let $s_0>d/2$.
The following holds.

\noindent
$(i)$  For $a\in H^{s_0}(\T^{d};\C)$ and for any $s\in \mathbb{R}$, one has
\begin{equation}\label{actionSob}
|T_{a}h|_{s}\lessdot_{s,s_0} |a|_{s_0}|h|_{s}\,,
\qquad \forall h\in H^{s}(\mathbb{T}^{d};\mathbb{C})\,.
\end{equation}

\noindent
$(ii)$ For any $\rho\in\mathbb{N}$, $s\in \R$,  and 
$a\in H^{s_0+\rho}(\T^{d};\C)$, 
 one has
 \begin{equation}\label{diffQuanti}
 |R^{\epsilon_1,\epsilon_2}_a h|_{{s+\rho}}\lessdot_{s,\rho,s_0}
 |h|_{s}|a|_{{\rho+s_0}}\,,
 \qquad \forall h\in H^{s}(\mathbb{T}^{d};\mathbb{C})\,.
 \end{equation}
 \end{lemma}
 
 \begin{proof}
 \emph{Item} $(i)$.
 First of all, using \eqref{cutofffunct}-\eqref{cutofffunctepsilon} and
since $0<\epsilon<1/4$,
 we note that, for $\x,\eta\in \mathbb{Z}^d$,
\begin{equation}\label{equixieta}
\chi_{\epsilon}\left(\frac{|\x-\eta|}{\langle \eta\rangle}\right)\neq0
\qquad \Rightarrow \qquad
\left\{
\begin{aligned}
(1+\tilde{\e})|\eta|&\geq |\x|\\
(1-\tilde\epsilon)|\eta|&\leq |\x| \,,
\end{aligned}
\right.
\end{equation}
where $0<\tilde{\epsilon}<2/5$.
Therefore, by definition of the norm
\begin{equation}\label{natale2}
\begin{split}
|T_{a}h|^{2}_{s}&
%\stackrel{\mathclap{\eqref{normaJapJap}}}
{\leq}
\sum_{\x\in\mathbb{Z}^{d}}
\jjap{\x}^{2s}\Big|
\sum_{\eta\in \mathbb{Z}^{d}}
\chi_{\epsilon}\left(\frac{|\x-\eta|}{\langle \eta\rangle}\right)
\widehat{a}(\x-\eta)\widehat{h}(\eta)\Big|^{2}
%\\&
\stackrel{\eqref{equixieta}}{\lessdot_s}
\sum_{\x\in\mathbb{Z}^{d}}
\Big(
\sum_{\eta\in \mathbb{Z}^{d}}
|\widehat{a}(\x-\eta)||\widehat{h}(\eta)|\jjap{\eta}^{s}\Big)^{2}
\\&
\lessdot_{s}\|(\widehat{h}(\x)\jjap{\x}^{s})\star \widehat{a}(\x)\|^{2}_{\ell^{2}}
\lessdot_{s}\|(\widehat{h}(\x)\jjap{\x}^{s})\|^{2}_{\ell^{2}}
\|\widehat{a}(\x)\|^{2}_{\ell^{1}}
\end{split}
\end{equation}
where we denoted by $\star$ the convolution between sequences
and where 
in the last inequality we used Young's inequality for 
sequences\footnote{For $f\in\ell^p(\C^d),\, g\in\ell^q(\C^d)$ 
where $1\le p,q,r \le \infty $ satisfy 
$\frac{1}{p} + \frac{1}{q} = \frac{1}{r} + 1$, then $\| f\star g\|_{\ell^r} \le \| f \|_{\ell^p} \| g \|_{\ell^q}$}. 
By Cauchy-Schwarz inequality, and using that $s_0>d/2$
we also deduce that
\begin{equation}\label{cespuglio1}
\|\widehat{a}(\x)\|_{\ell^{1}(\Z^{d})}\leq 
\|\widehat{a}(\x)\jjap{\x}^{s_0} \|_{\ell^{2}(\Z^{d})}\mathtt{c}(s_0)
\stackrel{\eqref{normaJapJap}}{=} 
\mathtt{c}_{s_0}|a|_{s_0}\,,
\qquad 
\mathtt{c}(s_0):=\Big(\sum_{\x\in\Z^{d}}\frac{1}{\jjap{\x}^{2s_0}}\Big)^{1/2}<+\infty\,.
\end{equation}
Then the latter bound, together with \eqref{natale2}, implies 
$|T_{a}h|^{2}_{s}\lessdot_{s,s_0}|a|_{s_0}^{2}|h|_{s}^{2}$, 
and  hence inequality \eqref{actionSob}.
 
 \noindent
 \emph{Item} $(ii)$.
 Notice that the set of $\x,\eta$ such that
$(\chi_{\epsilon_1}-\chi_{\epsilon_2})(|\x-\eta|/\langle\eta\rangle)=0$
contains the set  such that
\[
|\x-\eta|\geq \frac{8}{5}\epsilon_1 \langle \eta\rangle\quad {\rm or}\quad
|\x-\eta|\leq \frac{5}{4}\epsilon_2 \langle\eta\rangle\,.
\]
Therefore $(\chi_{\epsilon_1}-\chi_{\epsilon_2})(|\x-\eta|/\langle\x+\eta\rangle)\neq0$
implies
\begin{equation}\label{condizio}
\frac{5}{4}\epsilon_2 \langle\eta\rangle\leq |\x-\eta|\leq \frac{8}{5}\epsilon_1 \langle\eta\rangle\,.
\end{equation}
For $\x\in \mathbb{Z}^{d}$ 
we denote $\mathcal{A}(\x)$ the set of 
$\eta\in \mathbb{Z}^{d}$ such that the \eqref{condizio}
holds. 
For $\eta\in \mathcal{A}(\x)$ we also note that
\begin{equation}\label{cespuglio}
|\x|\leq c_1|\eta|\;\; {\rm and}\;\; |\x|\leq c_2|\x-\eta| 
\qquad \Rightarrow \qquad
\jjap{\x}\leq c_1\jjap{\eta}\;\; {\rm and}\;\; \jjap{\x}\leq c_2\jjap{\x-\eta}\,,
\end{equation}
for some absolute constants $c_1,c_2>0$.
To estimate the remainder in \eqref{natale}
we reason as in \eqref{natale2}.
By \eqref{condizio} %and setting $\rho=s-s_0$ 
we have
\begin{equation}\label{stimarestoresto}
\begin{aligned}
|R^{\epsilon_1,\epsilon_2}_ah|_{s+\rho}^{2}
&\stackrel{{\eqref{normaJapJap},\eqref{natale}}}{\leq}\sum_{\x\in\mathbb{Z}^{d}}
\jjap{\x}^{2(s+\rho)}\Big| 
\sum_{\eta\in \mathcal{A}(\x)}(\chi_{\epsilon_1}-\chi_{\epsilon_2})
\left(\frac{|\x-\eta|}{\langle\eta\rangle}\right)
\widehat{a}(\x-\eta)\widehat{h}(\eta)\Big|^{2}
\\
&\stackrel{\mathclap{\eqref{cespuglio}}}{\lessdot_{s,\rho}}
\sum_{\x\in \mathbb{Z}^{d}}
\Big(\sum_{\eta\in \mathcal{A}(\x)}
|\widehat{a}(\x-\eta)|\jjap{\x-\eta}^{\rho}
|\widehat{h}(\eta)|\jjap{\eta}^{s}\Big)^{2}
\\&
\lessdot_{s,\rho}\|(|\widehat{h}(\x)|\jjap{\x}^{s})\star 
(|\widehat{a}(\x)|\jjap{\x}^{\rho})\|^{2}_{\ell^{2}(\Z^{d})}
\lessdot_{s,\rho}
\||\widehat{h}(\x)|\jjap{\x}^{s}\|^{2}_{\ell^{2}(\Z^{d})}
\||\widehat{a}(\x)|\jjap{\x}^{\rho}\|^{2}_{\ell^{1}(\Z^{d})}
\\&
\stackrel{\eqref{cespuglio1}}{\lessdot_{s,\rho,s_0}}|a|^{2}_{s_0+\rho}|h|_{s}^{2}\,,
\end{aligned}
\end{equation}
where again we used Young inequality for convolutions.
The latter bound implies \eqref{diffQuanti}. 
 \end{proof}
 
As we have just proved,  the choice of the cut-off function $\chi_\epsilon$ 
is irrelevant in the definition of $T^\epsilon_a$,  hence we shall drop $\epsilon$ 
from the notation and write $T_a$ only.  

Let us now fix $s_0>d/2$ and 
let the symmetric cut-off function
\begin{equation}\label{alexander}
\psi(v,w) := 1-\chi_{\e}\Big(\frac{|v|}{\langle w\rangle}\Big)
-\chi_{\e}\Big(\frac{|w|}{\langle v\rangle}\Big),\,\quad v,w\in\R^{d}\,,
\end{equation}
be set.
For $a\in C^{\infty}(\mathbb{T}^{d};\mathbb{C})$  consider the  the map
\begin{equation}\label{mappaQQ}
\mathcal{F}(a)[h]:=\frac{1}{(2\pi)^{d}}\sum_{\x\in\Z^{d}}e^{\ii \x\cdot x}
\sum_{\eta\in \mathbb{Z}^{d}}
\psi(j-\eta,\eta)\widehat{a}(j-\eta)\widehat{h}(\eta)\,,
\quad 
h\in C^{\infty}(\mathbb{T}^{d};\mathbb{C})\,.
\end{equation}
We have the following.
\begin{lemma}\label{extensionResto} 
Let $s_0>d/2$. For any $\rho\geq0$
the map in \eqref{mappaQQ} extends as a continuous map
\[
\begin{aligned}
H^{s_0+\rho}(\mathbb{T}^{d};\mathbb{C})&\rightarrow 
\mathcal{L}\big(H^{s}(\mathbb{T}^{d};\mathbb{C}); H^{s+\rho}(\mathbb{T}^{d};\mathbb{C}) \big)\,,
\qquad s>0\,,
\\
a &\mapsto \mathcal{F}(a)[\cdot]
\end{aligned}
\]
and there is $C=C(\rho)>0$ such that for any $s>0$
\begin{equation}\label{stimaprecisaQQ}
|\mathcal{F}(a)[h]|_{s+\rho}\leq C^{\max\{s,s_0\}}|a|_{s_0+\rho}|h|_{s}\,,
\qquad 
\forall\, h\in H^{s}\,. 
\end{equation}
\end{lemma}

\begin{proof}
Recalling \eqref{alexander} we notice that, for any $\x,\eta\in \Z^{d}$,
\begin{equation}\label{cutTheta2}
\psi(\xi-\eta,\eta)\neq0 \qquad \Rightarrow\qquad
C^{-1}|\eta|\leq |\xi-\eta|\leq C|\eta|\,,
\end{equation}
for some absolute constant $C>0$, hence we deduce the equivalence between $\xi - \eta$ and $\eta$.
Then, using \eqref{mappaQQ}, we get
\[
\begin{aligned}
|\mathcal{F}(a)[h]|_{s+\rho}&\leq\Big(\sum_{\x\in\mathbb{Z}^{d}}\Big(
\sum_{\eta\in\mathbb{Z}^{d}}
|\psi(\x-\eta,\eta)||\widehat{a}(\x-\eta)||\widehat{h}(\eta)|
\jjap{\x}^{s+\rho}
\Big)^{2}\Big)^{1/2}\\
&\stackrel{\eqref{cutTheta2}}{\lessdot_{s,\rho}}
\|(|\widehat{a}(j)|\jjap{j}^{\rho})\star(|\widehat{h}(j)|\jjap{j}^{s})\|_{\ell^{2}(\Z^{d})}
\\&
\lessdot_{s,\rho}
\|(|\widehat{a}(j)|\jjap{j}^{\rho})\|_{\ell^{1}(\Z^{d})}
\|(|\widehat{h}(j)|\jjap{j}^{s})\|_{\ell^{2}(\Z^{d})}
%\\&
\lessdot_{s,\rho,s_0}|a|_{s_0+\rho}|h|_{s}\,,
\end{aligned}
\]
where we used that (since $s>0$, $\rho\geq0$) 
\[
|\x|\lessdot |\x - \eta| \;\;\;\Rightarrow\;\;\; 
|\x|^{s+\rho}\lessdot_{s,\rho} |\x-\eta|^\rho |\eta|^{s}
\]
Young's and Cauchy-Schwarz inequalities, 
and the fact that $s_0>d/2$, $s>0$. This implies \eqref{stimaprecisaQQ}.
\end{proof}

As a consequence of the result above we obtain the following.
 \begin{lemma}{\bf (Products).}\label{lem:prodotto}
Let $s\geq s_0>d/2$.
The product of any $a,b\in H^{s}(\T^{d};\C)$ can be equivalently expressed either as\\
1) \begin{equation*}%\label{eq:paraproduct}
a\cdot b = 
 T_{a}b + T_{b}a + \mathcal{Q}(a)[b]\,,
\end{equation*}
where the remainder $\mathcal{Q}(a)[b]=\mathcal{F}(a)[b]$ (see \eqref{mappaQQ}) 
and in particular satisfies, 
for any $0\leq \rho\leq s-s_0$,
the bound
\begin{equation}\label{eq:paraproduct22 je}
|\mathcal{Q}(a)[b]|_{H^{s+\rho}}\lessdot_{s,\rho,s_0}
%|a|_{s}|b|_{s_0+\rho}+
|a|_{s_0+\rho}|b|_{s}\,,
%\qquad \forall \, b\in H^{s}(\T^{d};\C)\,.
\end{equation}
or as \\
2) \begin{equation*}%\label{eq:paraproduct je}
a\cdot b = 
 T_{a}b + T_{b}a + \mathcal{Q}(b)[a]\,,
\end{equation*} 
\begin{equation*}
\widehat{\mathcal{Q}(b)[a]}(j)=\frac{1}{(2\pi)^{d}}
\sum_{\eta\in \mathbb{Z}^{d}}
\psi(j-\eta,\eta)\widehat{b}(j-\eta)\widehat{a}(\eta)\,,
\end{equation*}
where for any $0\leq \rho\leq s-s_0$
the following bound holds
\begin{equation*}%\label{eq:paraproduct22 je bis}
|\mathcal{Q}(b)[a]|_{H^{s+\rho}}\lessdot_{s,\rho,s_0}
|b|_{s_0+\rho}|a|_{s}\,.
\end{equation*}
\end{lemma}

\begin{proof}
Recall formul\ae\, \eqref{paraprod} and \eqref{cutofffunct}-\eqref{cutofffunctepsilon}
and let us define
\[
\mathcal{R}:=\mathcal{R}(a,b):=a\cdot b-T_{a}b-T_{b}a\,,\quad 
\widehat{\mathcal{R}}(j):=\sum_{\eta\in\Z^{d}}\Theta(j,\eta)\widehat{a}(j-\eta)\widehat{b}(\eta)\,,
\quad \forall j\in\Z^{d}
\]
where $\Theta : \R^{2}\to[-1,1]$ is the cut-off function defined as (see \eqref{alexander})
$\Theta(\xi,\eta):=\psi(\x-\eta,\eta)$.
%\begin{equation}\label{cutTheta}
%\Theta(\xi,\eta):=\psi(\x-\eta,\eta)\,.
%\end{equation}
The equivalence between the frequencies $\x-\eta$ and $\eta$
induced by \eqref{cutTheta2}
allows us to unload the strongest norm either on $a$ or on $b$, 
and obtain the desired inequalities provided that $s_0+\rho\leq s$ and 
using Lemma \ref{extensionResto}.
\end{proof}

\begin{lemma}{\bf (Tame estimate).}\label{lem:tametame}
Let $s\geq s_0>d/2$. Then there exists an absolute  constant $\tC>0$
such that, for any $a,b\in H^{s}(\mathbb{T}^{d};\mathbb{C})$, one has
\begin{equation}\label{tameestimate1}
|ab|_{s}\leq 
\mathtt{C}^{s}\big(
|a|_{s}|b|_{s_0}+|a|_{s_0}|b|_{s}\big)\,.
\end{equation}
\end{lemma}

\begin{proof}
Since $a\cdot b=T_ab+T_ba +\cQ(a)[b]$,  
the bound follows directly from 
Lemmata  \ref{azioneSimboo} and \ref{lem:prodotto} (applied with $\rho=0$).
\end{proof}

 \subsection{Compositions}\label{subparacompo}
 We now discuss compositions among para-products, Fourier multipliers and remainders.

\begin{lemma}{\bf (Commutators).}\label{lem:commutator}
Consider the operators $\langle D\rangle$ and $\jjap{D}$  defined in \eqref{langlesimbolo}
and \eqref{def:japjapModificato} respectively.
Let $s_0>d/2$, $m\in\R$ and $a\in H^{s_0}(\T^{d};\C)$.
Then for any $s\in \R$ one has
\begin{equation}\label{cespuglio3}
|\langle D\rangle^{-m}\circ T_{a} h|_{s+m}+
|T_{a}\circ\langle D\rangle^{-m} h|_{s+m}\lessdot_{s,s_0} R^{|m|}|a|_{s_0}|h|_s\,,\qquad 
\forall\, h\in H^{s}(\T^{d};\C)\,.
\end{equation}
If in addition $a$ belongs to $H^{s_0+1}(\T^{d};\C)$, then for $s\in \R$ and $p>0$, 
one has
\begin{equation}\label{cespuglio4}
\big| \big[\jjap{D}^{p} ,T_a\big] h\big|_{s+1-p}\leq R^p\mathtt{C}^{\max\{s,s_0\}+p}|a|_{s_0+1}|h|_{s}\,,\qquad 
\forall\, h\in H^{s}(\T^{d};\C)\,,
\end{equation}
for some absolute constant $\mathtt{C}>0$.
\end{lemma}

\begin{proof}
Let us start form bound \eqref{cespuglio3}. We have
\[
\begin{aligned}
|\langle D\rangle^{-m}\circ T_{a} h|_{s+m}^{2}&\leq
\sum_{j\in\mathbb{Z}^{d}}
\jjap{j}^{2(s+m)}\Big|
\sum_{\eta\in \mathbb{Z}^{d}}
\chi_{\epsilon}\left(\frac{|j-\eta|}{\langle \eta\rangle}\right)
\widehat{a}(j-\eta)\widehat{h}(\eta)\langle j\rangle^{-m}\Big|^{2}
\\
&\stackrel{\eqref{equixieta}}{\lessdot_s}
\sum_{j\in\mathbb{Z}^{d}}
\Big(
\sum_{\eta\in \mathbb{Z}^{d}}
|\widehat{a}(j-\eta)||\widehat{h}(\eta)|\jjap{\eta}^{s}
\Big(\frac{\jjap{j}}{\langle j\rangle}\Big)^{m}
\Big)^{2}\,.
\end{aligned}
\]
Notice that 
$\jjap{j}/\langle j\rangle\leq R$ and $\langle j\rangle/\jjap{j}\leq 2$. Therefore,
reasoning as in Item $(i)$ of Lemma \ref{azioneSimboo},
we deduce \eqref{cespuglio3} for the operator $\langle D\rangle^{-m}\circ T_{a}$.
The other estimate is similar.
Let us now check the bound \eqref{cespuglio4}.
First of all, by using \eqref{def:japjapModificato} and \eqref{cutofffunct}-\eqref{cutofffunctepsilon},
one can easily see that,
for any $j,\eta\in\Z^{d}$ such that $\chi_{\e}(|j-\eta|/\langle \eta\rangle)\neq0$,
one has
\begin{equation*}%\label{cespuglio5}
\big|\jjap{j}^{p}-\jjap{\eta}^{p}\big|
\leq C |j-\eta| \jjap{j}^{-p + 1}\max\{|j|,|\eta|\}^{p-1}\leq R^p \tilde{C}^{p}|j-\eta|\,,
\end{equation*}
 for some  constants  $\tilde{C}\geq C>1$.
 We hence have
 \begin{equation*}
 \begin{aligned}
 \big| \big[\jjap{D}^{p} ,T_a\big] h\big|_{s+1-p}^{2}&\leq
 \sum_{j\in\mathbb{Z}^{d}}
\Big(
\sum_{\eta\in \mathbb{Z}^{d}}
\chi_{\epsilon}\left(\frac{|j-\eta|}{\langle \eta\rangle}\right)
|\widehat{a}(j-\eta)||\widehat{h}(\eta)|\jjap{j}^{s}
\frac{\big|\jjap{j}^{p}-\jjap{\eta}^{p}\big|}{\jjap{j}^{p-1}}
\Big)^{2}
\\&
\lessdot_{s}
%\stackrel{\eqref{cespuglio5}}{\lessdot_{s}} 
\widetilde{C}^{2p}\sum_{j\in\mathbb{Z}^{d}}
\Big(
\sum_{\eta\in \mathbb{Z}^{d}}
|\widehat{a}(j-\eta)| |j-k| |\widehat{h}(\eta)|\jjap{\eta}^{s}
\Big)^{2}
\\&
\lessdot_{s} \widetilde{C}^{2p}\|(|\widehat{a}(j)|\jjap{j} \star(|\widehat{h}(j)|\jjap{j}^{s})\|_{\ell^{2}(\Z^{d})} \,.
 \end{aligned}
 \end{equation*}
 Hence, reasoning as done in item $(i)$ of Lemma \ref{azioneSimboo},
 we get \eqref{cespuglio4}.
\end{proof}

We have the following.
\begin{lemma}{\bf (Composition 1).}\label{lem:composition}
Let $s_0>d/2$, $m_1,m_2\geq0$ and consider 
$a,b\in H^{s_0}(\T^{d};\C)$.
Then for any $s\in \R$ one has
\begin{equation*}%\label{cespuglio10}
|T_{a}\circ\langle D\rangle^{-m_1}\circ T_{b}
\circ\langle D\rangle^{-m_2}h|_{s+m_1+m_2}\lessdot_{s,s_0}
R^{m_1 + m_2}|a|_{s_0}|b|_{s_0}|h|_{s}\,,
\qquad
\forall\, h\in H^{s}(\T^{d};\C)\,.
\end{equation*}
\end{lemma}

\begin{proof}
We shall apply iteratively Lemma \ref{lem:commutator}.
By setting $R:= T_{a}\circ\langle D\rangle^{-m_1}\circ T_{b}
\circ\langle D\rangle^{-m_2} $
we have that
\[
\begin{aligned}
|Rh|_{s+m_1+m_2}
\stackrel{\eqref{cespuglio3}}{\lessdot_{s,s_0}}
R^{m_1}|a|_{s_0}
|T_{b}\circ\langle D\rangle^{-m_2}h|_{s}
\stackrel{\eqref{cespuglio3}}{\lessdot_{s,s_0}}
R^{m_1+m_2}|a|_{s_0}|b|_{s_0}|h|_{s}\,.
\end{aligned}
\]
This concludes the proof.
\end{proof}

\begin{lemma}{\bf (Compositions 2).}\label{compoparapara}
Let $s_0>d/2$, $\rho\geq 0$ and consider $a,b\in H^{s_0+\rho}(\T^{d};\C)$.
Then one has
\[
T_{a}\circ T_{b}[\cdot] = T_{ab}[\cdot] + \mathcal{Q}(a,b)[\cdot],
\]
where the remainder $ \mathcal{Q}(a,b)$ satisfies 
\[
\begin{aligned}
H^{s_0+\rho}(\mathbb{T}^{d};\mathbb{C})\times H^{s_0+\rho}(\mathbb{T}^{d};\mathbb{C})
&\rightarrow 
\mathcal{L}\big(H^{s}(\mathbb{T}^{d};\mathbb{C}); H^{s+\rho}(\mathbb{T}^{d};\mathbb{C}) \big)\,,
\qquad s>0\,,
\\
(a,b) &\mapsto \mathcal{Q}(a,b)[\cdot]
\end{aligned}
\]
and
\begin{equation}\label{calma1}
|\mathcal{Q}(a,b)[h]|_{s+\rho}\lessdot_{s,\rho,s_0}
|a|_{s_0+\rho}|b|_{s_0+\rho}|h|_{s}\,,
\qquad 
\forall\, h\in H^{s}(\T^d;\C)\,.
\end{equation}
\end{lemma}

\begin{proof}
Recalling \eqref{paraprod} and setting $\mathcal{Q}=\mathcal{Q}(a,b)[\cdot]$ we note that
\[
\widehat{(\mathcal{Q}h)}(j)=
\sum_{\eta,\theta\in\Z^{d}}
\mathtt{r}(j,\theta,\eta)
\widehat{a}(j-\theta) \widehat{b}(\theta-\eta)
\widehat{h}(\eta)\,,
\]
where, for any $j,\eta,\theta\in\Z^{d}$ we defined
\[
\mathtt{r}(j,\theta,\eta)
:=\chi_{\e}
\Big(\frac{|j-\theta|}{\langle \theta\rangle}\Big)
\chi_{\e}\Big(
\frac{|\theta-\eta|}{\langle \eta\rangle}
\Big)
-\chi_{\e}\Big(\frac{|j-\eta|}{\langle \eta\rangle}\Big)
\]
One can check that , for instance,
\begin{equation*}%\label{hansolo}
\mathtt{r}(j,\theta,\eta)\neq0
\quad \Rightarrow \quad 
|j-\eta|\leq \frac{8}{5}\e|\eta|\,,\;\;
|\eta-\theta|\leq \frac{8}{5}\e|\theta|\,,\;\;
|j-\theta|\geq \frac{5}{4}\e|\theta|\,.
\end{equation*}
The above inequalities imply the equivalence of $|j|,\,|\eta|$, and $|\theta - j|$, 
with appropriate pure constants, which yields the following chain of inequalities.
From the definition of the norm we get
\begin{equation*}
| \widehat{(\mathcal{Q}h)}(j) |\jjap{j}^{s+\rho}
\lessdot_{s}
\sum_{\eta,\theta\in\Z^d}
\jjap{j-\theta}^{\rho}|\widehat{a}(j-\theta)|
| \widehat{b}(\theta-\eta) |
| \widehat{h}(\eta) |\jjap{\eta}^{s}\,,\qquad \forall j,\theta,\eta\in\Z^{d}\,.
\end{equation*}
which, together with Young's inequality, implies
\[
\begin{aligned}
|\mathcal{Q}h|_{s+\rho}&=
\|\widehat{(\mathcal{Q}h)}(j)\jjap{j}^{s+\rho}\|_{\ell^{2}(\Z)}
\lessdot_s
\|
(\widehat{a}(j)\jjap{j}^{\rho})\star
(\widehat{b}(j))\star
(\widehat{h}(j)\jjap{j}^{s})
\|_{\ell^{2}(\Z^d)}
\\&\lessdot_{s}
\|(\widehat{a}(j)\jjap{j}^{\rho})\|_{\ell^{1}(\Z^d)}\|
(\widehat{b}(j))\|_{\ell^{1}(\Z^{2})}
\|\widehat{h}(j)\jjap{j}^{s}\|_{\ell^{2}(\Z^{2})}
\lessdot_{s,s_0}|a|_{s_0+\rho}|b|_{s_0+\rho}|h|_{s}\,,
\end{aligned}
\]
which is the estimate \eqref{calma1}.
\end{proof}

\begin{lemma}{\bf (Compositions 3).}\label{compoparasmooth}
Let $\rho\geq0$, $s_0>d/2$, %$s\geq s_0+\rho$ 
and consider a map
\begin{equation*}
\begin{aligned}
\mathcal{Q} \; :\; B_{1}(H^{s_0+\rho}(\T^{d};\C))&\longrightarrow 
\mathcal{L}\big(H^{s}(\T^{d};\C); H^{s+\rho}(\T^{d};\C)\big)\,,\;\;\;s>0
\\&
\!\!\!\!b \;\;\mapsto\;\; \mathcal{Q}(b)
\end{aligned}
\end{equation*}
satisfying 
\begin{equation}\label{calma4}
|\mathcal{Q}(b)[h]|_{s+\rho}\lessdot_{s,\rho}|b|_{s_0+\rho}|h|_{s}\,,\qquad 
\forall\, h\in H^{s}(\T^{d};\C)\,,\;\;\;s>0\,.
\end{equation}
Consider also $a\in H^{s_0}$. Then one has
\begin{equation}\label{calma3}
|T_{a}\circ \mathcal{Q}(b)[h]|_{s+\rho}
+|\mathcal{Q}(b)\circ T_a [h]|_{s+\rho}
\lessdot_{s,\rho,s_0}|b|_{s_0+\rho}|a|_{s_0}
|h|_{s}\,,
\end{equation}
for any $h\in H^{s}(\T^{d};\C)$.
\end{lemma}
\begin{proof}
The bound \eqref{calma3} follows by combining estimates 
\eqref{calma4} and \eqref{actionSob}.
\end{proof}

For $N\geq1$, $a\in H^{s}(\T^d;\C)$ we define the
operators 
$T^{\leq N}_{a}$ and $T_{a}^{>N}$ as (recall Definition \ref{def:projectors})
\begin{equation}\label{def:truncat}
T_{a}^{\leq N}=T_{a}\circ \Pi_{N}\,,\qquad T_{a}^{>N}:=T_{a}\circ\Pi_{N}^{\perp}=
T_{a}^{>N}\circ(\uno-\Pi_{N})\,. 
\end{equation}
We now show that $T_{a}^{\leq N}$ is a \emph{regularizing} operator,
while the operator $T_{a}^{>N}$ ``can transform'' smoothing effects into smallness.
More precisely we have the following.
\begin{lemma}{\bf (Truncations).}\label{danguard}
Let $s_0>d/2$, $m\geq 1$, $N\geq R$ (see \eqref{def:japjapModificato})
and $a\in H^{s_0}(\T^{d};\C)$.
Then for any $s>0$ one has
\begin{equation}\label{calma20}
|T_{a}^{\leq N}h|_{s+\beta}\lessdot_{s,s_0} N^{\beta}|a|_{s_0}|h|_{s}\,,
\qquad 
\forall \, h\in H^{s}(\T^{d};\C)\,,\;\; \beta\geq0\,.
\end{equation}
Moreover (recall \eqref{langlesimbolo})
\begin{align}
|T_{a}^{>N}\circ \langle D\rangle^{-m}h|_{s+m}
&\lessdot_{s,s_0} R^m |a|_{s_0}|h|_{s}\,,
\label{calma21}
\\
|T_{a}^{>N}\circ \langle D\rangle^{-m}h|_{s+m-1}
&\lessdot_{s,s_0} R^m N^{-1} |a|_{s_0}|h|_{s}
\label{calma22}
\end{align}
for any $h\in H^{s}(\T^{d};\C)$.
The same holds for the operator $\langle D\rangle^{-m}\circ T_{a}^{>N}$.
\end{lemma}

\begin{proof}
Let us check the first bound for some $\beta\geq0$.
One has
\[
|T_{a}^{\leq N}h|_{s+\beta}
\stackrel{\eqref{def:truncat},\eqref{actionSob}}{\lessdot_{s,s_0}}
|a|_{s_0}|\Pi_{N}h|_{s+\beta}
\lessdot_{s,s_0}
N^{\beta}|a|_{s_0}|h|_{s}\,,
\]
where we used the first bound in Lemma \ref{interopolo}.
This implies  \eqref{calma20}. Moreover
we have
\[
\begin{aligned}
|T_{a}^{>N}\circ \langle D\rangle^{-m}h|_{s+m-1}
&\stackrel{\eqref{def:truncat}}{=}
|T_{a}\circ \langle D\rangle^{-m}(\Pi_{N}^{\perp}h)|_{s+m-1}
\\&\stackrel{\eqref{cespuglio3}}{\lessdot_{s,s_0}}
R^m|a|_{s_0}|\Pi_{N}^{\perp}h|_{s-1}
\lessdot_{s,s_0}
R^m N^{-1}|a|_{s_0}|h|_{s}\,,
\end{aligned}
\]
where we also used 
Lemma \ref{interopolo}.
This proves the \eqref{calma22}.
The bound \eqref{calma21} follows trivially.
\end{proof}

\section{Paralinearization of NLS and nonlinear estimates}\label{sec:paraNLS}

In this section we provide paralinearization of the nonlinearity 
$|u|^{2p}u$ of the equation \eqref{QLNLS}.
We first give the following Definition.
\begin{definition}{\bf (p-admissible).}\label{def:pAdmissible}
Let $\rho\geq0$, $s_0>d/2$, %$s\geq s_0+\rho$, 
$p=k/2$ with  $k\in \N$.
We say that  a map %as in \eqref{mappaQQ}
is $(p,\rho)$-\emph{admissible} if 
\begin{equation*}
\begin{aligned}
\mathcal{Q} \; :\; B_{1}(H^{s_0+\rho}(\T^{d};\C))&\longrightarrow 
\mathcal{L}\big(H^{s}(\T^{d};\C); H^{s+\rho}(\T^{d};\C)\big)\,,\;\;\;s>0
\\&
\!\!\!\!w \;\;\mapsto\;\; \mathcal{Q}(w)
\end{aligned}
\end{equation*}
and 
there exists a constant $C:=C(\rho)>0$ such that 
for any $s>0$ and 
%$s\geq s_0+\rho$ and 
$w\in B_1(H^{s_0+\rho}(\T^{d};\C))$
one has
\begin{equation}\label{mortenera}
|\mathcal{Q}(w)h|_{s+\rho}\leq C^{2p\max\{s,s_0\}} |w|_{s_0+\rho}^{2p} |h|_{s}\,,
\qquad 
\forall \, h\in H^{s}(\T^{d};\C)\,. 
\end{equation}
We say that a map $\mathcal{Q}$ is $p$-\emph{admissible}
if for any $\rho\geq0$ it is $(p,\rho)$-admissible.
\end{definition}
The key result of this section is the following.

\begin{proposition}\label{prop:paralinp}
Let  $p\geq 1$, $ s_0>d/2$ and $u\in H^{s_0}$. One has
\begin{equation}\label{paralin:gradop}
|u|^{2p}u=(p+1)T_{|u|^{2p}}[u]+pT_{|u|^{2(p-1)}u^2}[\bar{u}]+\mathcal{Q}_1^{(p)}(u)[u]
+\mathcal{Q}_2^{(p)}(u)[\bar{u}]\,,
\end{equation}
where $\mathcal{Q}_{j}^{(p)}$, $j=1,2$, are regularizing operators
which are $p$-admissible according to Definition \ref{def:pAdmissible}.
In particular, for any $\rho\geq0$, if $u\in H^{s}$ with $s\geq s_0+\rho$
 then
 \[
 |\mathcal{Q}_j^{(p)}(u)[u]|_{s+\rho}\lessdot_{s,\rho} |u|_{s_0+\rho}^{2p}|u|_{s}\,,\;\;\; j=1,2\,.
 \]
\end{proposition}
For the sake of clarity we 
first need a basic \emph{tame} estimate.
\begin{lemma}\label{lem:confort2}
For any $s\geq s_0>d/2$ there exists $\mathtt{M}>0$ such that,
for all integer $p\geq1$ one has
\begin{equation}\label{confort2}
| |u|^{2p}u |_{s}\leq \mathtt{M}^{2ps}|u|_{s_0}^{2p}|u|_{s}\,,
\qquad
| |u|^{2p} u|_{s}\leq \mathtt{M}^{2p s_0}
(\mathtt{M}/R)^{2p(s-s_0)}|u|_{s}^{2p+1}\,,
\end{equation}
for any $u\in H^{s}(\T^d;\C)$.
\end{lemma}

The result above  
 is a consequence of the following Lemma.
\begin{lemma}{\bf (Basic nonlinear estimates).}\label{stimaNonlin}
For any $s\geq s_0>d/2$ there exists an absolute constant $\mathtt{M}>0$ such that,
for all integers $q_1,q_2\geq0$, $q_1+q_2\geq1$,
\begin{equation}\label{confort}
| u^{q_1}\bar{u}^{q_2} |_{s}\leq \mathtt{M}^{(q_1+q_2-1) s}|u|_{s_0}^{q_1+q_2-1}|u|_{s}\,,
\qquad 
\forall\; u\in H^{s}(\T^d;\C)\,.
\end{equation}
\end{lemma}

\begin{proof}
We give the proof for the case $q_1=2p+1$, $q_2=0$, $p\geq0$, 
 i.e. we estimate $u^{2p+1}$.
The general case follows similarly.
We show, by induction on $p\geq0$,
that one has
\begin{equation}\label{stimaInd}
|u^{2p}u|_{s} \leq 8^{p}\mathtt{K}^{2p s}|u|_{s_0}^{2p}|u|_{s}\,,
\end{equation}
for some absolute $\mathtt{K}>0$.
The case $p=0$ is trivial.
Let us consider the case $p=1$.
First of all we have that estimate \eqref{tameestimate1}
for $s=s_0$ and $u=v$ reduces to
\[
|u^2|_{s_0}=|u\bar{u}|_{s_0}\leq 2\mathtt{C}^{s}|u|_{s_0}^{2}\,. 
\]
Then 
\[
|u^{2}u|_{s}\leq 2\tC^{s}\big( |u^{2}|_{s_0}|u|_{s}+
|u^{2}|_{s}|u|_{s_0}\big)
\leq 
2\tC^{s}\big( 
2\mathtt{C}^{s_0}+2\mathtt{C}^{s}
\big)|u|_{s_0}^{2}|u|_{s}\,,
\]
which implies \eqref{stimaInd} with $p=1$.
Now assume that \eqref{stimaInd} holds for some $p\geq2$.
Therefore
\[
\begin{aligned}
|u^{2(p+1)}u|_{s}&=
|u^{2p}\cdot u\cdot u^2|_{s}
\stackrel{\eqref{tameestimate1}}{\leq} 
2\tC^{s}\big( |u^{2p}u|_{s_0}|u^2|_{s}+
|u^{2p}u|_{s}|u^{2}|_{s_0}\big)
\\&
\stackrel{\eqref{stimaInd}}{\leq }
2\tC^{s}\Big(
8^{p+1}2\mathtt{C}^{s} \mathtt{K}^{2ps_0}+2\mathtt{C}^{s_0}8^{p}\mathtt{K}^{2ps}
\Big)
|u|_{s_0}^{2(p+1)}|u|_{s}
\leq 8^{p+1} \mathtt{K}^{2(p+1)s}|u|_{s_0}^{2(p+1)}|u|_{s}\,,
\end{aligned}
\]
which is \eqref{stimaInd} for $p\rightsquigarrow p+1$.
Then \eqref{confort} follows.
\end{proof}

\begin{proof}[{\bf Proof of Lemma \ref{confort2}}]
The bounds in \eqref{confort2} follow by \eqref{confort} together with \eqref{confort3}.
\end{proof}

\begin{proof}[{\bf Proof of Proposition \ref{prop:paralinp}}]
Following the ideas of Lemmata \ref{extensionResto}, \ref{lem:prodotto}.
we reason by induction on $p$. 

\noindent
{\bf Case:} $p=1$.
We claim that \textit{
for   $s\geq s_0>d/2$ one has
\begin{equation}\label{paralin:cubic}
|u|^{2}u=2T_{|u|^{2}}u+T_{u^2}\bar{u}+\mathcal{Q}_1^{(1)}(u)[u]+\mathcal{Q}_2^{(1)}(u)[\bar{u}]\,,
\end{equation}
where $\mathcal{Q}_{j}^{(1)}(u)$, $j=1,2$, are $1$-admissible operators.
}

By applying 
Lemma \ref{lem:prodotto} with $a=|u|^{2}$ and $b=u$
we get
\begin{equation}\label{calma11}
|u|^{2}u=T_{|u|^{2}}u+T_{u}\big[|u|^{2}\big]+\mathcal{R}_1(u)[u]\,,
\end{equation}
where $\mathcal{R}_1(u)[\cdot]:= \cF(|u|^2)[\cdot]$ is of the form \eqref{mappaQQ} and 
satisfies the properties in
 Lemma \ref{extensionResto}.
 In particular, for any $\rho\geq 0$ and any $w\in H^{s_0+\rho}$, estimate \eqref{stimaprecisaQQ}
 implies 
 \[
 |\mathcal{R}_1(w)[h]|_{s+\rho}{\lessdot_{s,\rho}}
|w^2|_{s_0+\rho}|h|_{s}
\stackrel{\eqref{tameestimate1}}{\lessdot_{s,\rho}}
|w|^2_{s_0+\rho}|h|_{s}\,,
 \]
 for any $h\in H^{s}(\T^{d};\C)$, $s>0$, i.e. $\mathcal{R}_1$ is $1$-admissible.

Similarly  by Lemmata \ref{extensionResto}, \ref{lem:prodotto} we deduce that 
\[
|u|^2=T_{\bar{u}}u+T_{u}\bar{u}+\mathcal{R}_2(u)[u]
\]
where $\mathcal{R}_2$ is $(1/2)$-admissible.
Hence the \eqref{calma11} becomes
\[
\begin{aligned}
|u|^2u&=T_{|u|^2}u+T_{u}\circ T_{\bar{u}}u+T_{u}\circ T_{u}\bar{u}
+T_{u}\circ \mathcal{R}_2(u)[u]+
\mathcal{R}_1(u)[u]\,.
\end{aligned}
\]
By Lemma \ref{compoparapara} we deduce
\[
T_{u}\circ T_{\bar{u}}u+T_{u}\circ T_{u}\bar{u}=
T_{|u|^2}u+T_{u^2}\bar{u}+\mathcal{R}_3(u)[u]+\mathcal{R}_4(u)[\bar{u}]
\]
where, 
by \eqref{calma1}, the remainders $\mathcal{R}_j(u)$, $j=3,4$, satisfy
the estimate \eqref{mortenera}, i.e. they are $1$-admissible.
Since $\mathcal{R}_2$ is $(1/2)$-admissible, estimate \eqref{mortenera} with $p=\tfrac{1}{2}$
guarantees (see \eqref{calma4})
that Lemma  \ref{compoparasmooth} applies to the 
operator
$\mathcal{R}_{5}(u):=T_u\circ \mathcal{R}_2(u)[\cdot]$.
Then the bound \eqref{calma3} (with $a\rightsquigarrow u$ and 
$\mathcal{Q}\rightsquigarrow\mathcal{R}_2(u)$)
implies \eqref{mortenera} for the operator $\mathcal{R}_5$.
The discussion above implies the 
paralinearization formula 
\eqref{paralin:cubic} and hence the thesis.

\noindent
{\bf Case:} $p\geq2$.
Assume now that formula \eqref{paralin:gradop} holds true for some $p\geq2$.
In particular (to fix the notation)
by Def. \ref{def:pAdmissible} we have that, for any $\rho\geq0$ there is an constant
$\mathtt{M}=\mathtt{M}(\rho)>0$ (depending only on $\rho$) such that for any 
 $w\in H^{ s_0+\rho}$ 
 \begin{equation}\label{stimaindu}
 |\mathcal{Q}_{j}^{(p)}(w)h|_{s+\rho}\leq \mathtt{M}^{2p\max\{s,s_0\}}|w|_{s_0+\rho}^{2p}|h|_{s}\,,
 \qquad
 \forall h\in H^{s}(\T^{d};\C)\,, s>0\,,\;\;\; j=1,2\,.
 \end{equation}
 We shall prove that $|u|^{2(p+1)}u=|u|^{2}(|u|^{2p}u)$ can be written as in 
 \eqref{paralin:gradop} with some remainders satisfying \eqref{stimaindu}
 with $p\rightsquigarrow p+1$
but with the same constant $\mathtt{M}>0$.
To do this we reason as follows.
We apply the para-product Lemmata \ref{extensionResto}, \ref{lem:prodotto}
with $a=|u|^{2p}u$ and $b=|u|^{2}$. So we get
\begin{equation}\label{mortenera2}
\begin{aligned}
|u|^{2(p+1)}u=T_{|u|^{2}}\big[|u|^{2p}u\big]+T_{|u|^{2p}u}\big[|u|^{2}\big]
+\mathcal{R}^{(1)}(u)\big[|u|^{2}\big]\,,
\end{aligned}
\end{equation}
where $\mathcal{R}^{(1)}(u)=\mathcal{F}(|u|^{2p}u)[\cdot]$ satisfies, for any $\rho\geq0$, %$s\geq s_0+\rho$, 
$w\in H^{s_0+\rho}(\T^{d};\C)$,
\begin{equation}\label{mortenera3}
|\mathcal{R}^{(1)}(w)[h]|_{s+\rho}\leq C_1^{2p\max\{s,s_0\}}|w|_{s_0}^{2p}|w|_{s_0+\rho}|h|_{s}\,,\qquad 
\forall\, h\in H^{s}(\T^{d};\C)\,, \;\;s>0
\end{equation}
for some constant $C_1=C_{1}(\rho)>0$ 
depending only on $\rho$\footnote{Notice that this constant $C_1$ has 
no relation with the constant $\mathtt{M}$
appearing in \eqref{stimaindu}}.
The bound \eqref{mortenera3} follows by \eqref{stimaprecisaQQ} 
with $a=|u|^{2p}u$
and using also \eqref{confort} to estimate the low norm of $a$.
In other words $\mathcal{R}^{(1)}$ is a $(p+\tfrac{1}{2})$-admissible remainder.
Indeed \eqref{mortenera3} implies \eqref{mortenera}
with $p\rightsquigarrow p+1/2$. 
We study each summand in \eqref{mortenera2} separately.

Notice that again by Lemmata \ref{extensionResto}, \ref{lem:prodotto} we deduce
\begin{equation}\label{mortenera4}
|u|^{2}=T_{\bar{u}}[u]+T_{u}[\bar{u}]+\mathcal{R}^{(2)}(u)[u]\,,
\end{equation}
where $\mathcal{R}^{(2)}$ (by estimate \eqref{eq:paraproduct22 je})
is a $\tfrac{1}{2}$-admissible remainder.
Hence, by applying Lemma \ref{compoparasmooth} and using 
\eqref{mortenera4},
we deduce that the third summand $\mathcal{R}^{(1)}(u)\big[|u|^{2}\big]$
in \eqref{mortenera2}
is a $(p+1)$-admissible remainder satisfying \eqref{mortenera} 
with $p\rightsquigarrow p+1$
for some constant $\mathtt{C}_1>0$.
Consider now the second summand in \eqref{mortenera2}. 
By applying Lemmata \ref{compoparapara} and \ref{compoparasmooth}
we get
\begin{equation}\label{mortenera5}
\begin{aligned}
T_{|u|^{2p}u}\big[|u|^{2}\big]&\stackrel{\eqref{mortenera4}}{=}
T_{|u|^{2p}u}\circ T_{\bar{u}}[u]+T_{|u|^{2p}u}\circ T_{u}[\bar{u}]
+T_{|u|^{2p}u}\circ \mathcal{R}^{(2)}(u)[u]
\\&
=T_{|u|^{2(p+1)}}[u]+T_{|u|^{2p}u^2}[\bar{u}]
+\mathcal{R}_1^{(3)}(u)[u]
+\mathcal{R}_2^{(3)}(u)[\bar{u}]
\end{aligned}
\end{equation}
where $\mathcal{R}_{j}^{(3)}$, $j=1,2$, are $(p+1)$-admissible remainders
satisfying \eqref{mortenera} 
with $p\rightsquigarrow p+1$
for some constant $\mathtt{C}_2>0$.
To study the first summand in \eqref{mortenera2}
we use the inductive assumption, i.e. $|u|^{2p}u$ has the form
\eqref{paralin:gradop} with remainders satisfying \eqref{stimaindu}. 
Hence, using Lemmata \ref{compoparapara} and \ref{compoparasmooth},
we can write
\begin{equation}\label{mortenera6}
\begin{aligned}
T_{|u|^{2}}&\big[|u|^{2p}u\big]=
T_{|u|^{2}}\Big[(p+1)T_{|u|^{2p}}[u]+pT_{|u|^{2(p-1)}u^2}[\bar{u}]+\mathcal{Q}_1^{(p)}(u)[u]
+\mathcal{Q}_2^{(p)}(u)[\bar{u}]\Big]
\\&
=(p+1)T_{|u|^{2(p+1)}}[u]+pT_{|u|^{2p}u^2}[\bar{u}]
+\mathcal{R}_1^{(4)}(u)[u]
+\mathcal{R}_2^{(4)}(u)[\bar{u}]
+\widetilde{\mathcal{Q}}_1(u)[u]
+\widetilde{\mathcal{Q}}_2(u)[\bar{u}]
\end{aligned}
\end{equation}
where $\mathcal{R}_{j}^{(4)}(u)$, $j=1,2$, are 
$(p+1)$-admissible remainders 
satisfying \eqref{mortenera} 
with $p\rightsquigarrow p+1$
for some constant $\mathtt{C}_3>0$,
while the operators
\[
\widetilde{\mathcal{Q}}_j(u)[\cdot]:=T_{|u|^{2}}\circ \mathcal{Q}_j^{(p)}(u)[\cdot]\,,
\quad j=1,2\,,
\]
are $(p+1)$-admissible remainders 
satisfying \eqref{mortenera}  (recall the inductive assumption \eqref{stimaindu})
with $p\rightsquigarrow p+1$
and with constant 
$\mathtt{C}_{4}^{2\max\{s,s_0\}} \mathtt{M}^{2p\max\{s,s_0\}}$.
Collecting together 
\eqref{mortenera2}, \eqref{mortenera5} and \eqref{mortenera6}
we can write
\[
|u|^{2(p+1)}u=
(p+2)T_{|u|^{2(p+1)}}[u]+(p+1)T_{|u|^{2p}u^2}[\bar{u}]
+\mathcal{Q}_1^{(p+1)}(u)[u]
+\mathcal{Q}_2^{(p+1)}(u)[\bar{u}]
\]
where, by the discussion above, the remainders
$\mathcal{Q}_j^{(p+1)}$, $j=1,2$, are 
$(p+1)$-admissible
and satisfy, for any $\rho\geq0$, %$s\geq s_0+\rho$, 
$w\in H^{s_0+\rho}(\T^{d};\C)$,
\begin{equation*}
|\mathcal{Q}^{(p+1)}_{j}(w)[h]|_{s+\rho}\leq
\big(
\mathtt{C}_1^{p\mathfrak{d}}+\mathtt{C}_2^{p\mathfrak{d}}+\mathtt{C}_3^{p\mathfrak{d}}+
\mathtt{C}_{4}^{\mathfrak{d}} \mathtt{M}^{p\mathfrak{d}}
\big)
 |w|_{s_0+\rho}^{2(p+1)}|h|_{s}\,,\qquad \forall \, h\in H^{s}(\T^{d};\C)\,,\;\;s>0\,,
\end{equation*}
where $\mathfrak{d}:=2\max\{s,s_0\}$.
Taking $\mathtt{M}=\mathtt{M}(\rho)$ large enough (only with respect 
to the constants $\mathtt{C}_i$ which depend only on $\rho$),
the latter bound implies 
\eqref{stimaindu} with $p\rightsquigarrow p+1$.
This concludes the proof.
\end{proof}

\section{A priori energy estimates}
In this section we provide a priori estimates on the evolution of the 
Sobolev norms of solutions of the equation \eqref{QLNLS}.
In particular from now on we work
on functions $u=u(t,x)$ of time and space satisfying the following 
assumption: 
\begin{hypo}{\bf (Local well-posedness).}\label{hyp:local}
Let $s\geq s_0>d/2$.
There exists $T>0$ such that for any 
$u_0\in H^{s}(\T^{d};\C)$
with $|u_0|_{s_0}\leq 1/2$,
there is a solution
$u(t,x)$
of \eqref{QLNLS}  with $u(0,x)=u_0(x)$ and 
\begin{equation}\label{hypo1}
\begin{aligned}
&u\in C^{0}\big([0,T];H^{s}(\T^{d};\C)\big)\cap
C^{1}\big([0,T];H^{s-2}(\T^{d};\C)\big)\,,
\\&
\sup_{t\in [0,T]}|u(t,\cdot)|_{s_0}\leq 1\,,\qquad 
\sup_{t\in [0,T]}|u(t,\cdot)|_{s}<+\infty\,.
\end{aligned}
\end{equation}
\end{hypo}
From now on we shall fix 
$R:=2\mathtt{M}$ where $\mathtt{M}>0$ is 
the constant  
provided by Lemma \ref{stimaNonlin} 
and we shall fix $T>0$ as the existence time given by 
the hypothesis above.

\subsection{Basic estimates}\label{sec:basicbasic}
One has the following.
\begin{theorem}{\bf (Basic a priori energy estimates).}\label{thm:energyBasic}
Let $s_0>d/2$ and $s_1\geq s_0$.
Then, if $u$ is a solution of \eqref{QLNLS} satisfying  
 Hypothesis \ref{hyp:local}, one has the a priori bound
 \begin{equation}\label{basicStima}
|u(t)|_{s_1}^{2}\leq |u_{0}|_{s_1}^{2}+\mathtt{M}^{2p s_0} 2^{-2p(s_1-s_0)}\int_{0}^{t}
|u(\s)|_{s_1}^{2p+2}d\s\,,\qquad t\in[0,T]\,,
\end{equation}
where $T>0$   is the one given in the local well posedness assumption Hyp. \ref{hyp:local}.
\end{theorem}
\begin{proof}
We have
\begin{equation*}
\begin{aligned}
\pa_{t}|u(t)|_{s_1}^{2}&\stackrel{\eqref{normaJapJap}}{=}
\pa_{t}(\jjap{D}^{s_1}u,\jjap{D}^{s_1}u)_{L^{2}}
\stackrel{\eqref{QLNLS}}{=}2{\Re} ( \jjap{D}^{s_1}\ii |u|^{2p}u,\jjap{D}^{s_1}u)_{L^{2}}
\\&\leq \|\jjap{D}^{s_1}|u|^{2p}u\|_{L^{2}}\|\jjap{D}^{s_1}u\|_{L^{2}}
\leq ||u|^{2p}u|_{s_1}|u|_{s_1}\,.
\end{aligned}
\end{equation*}
Therefore, by the second estimate in \eqref{confort2} in Lemma \ref{lem:confort2} 
and recalling that we set $R=2\mathtt{M}$, 
we get
\[
\pa_{t}|u(t)|^{2}_{s_1}\leq |u|_{s_1}^{2p+2}  \mathtt{M}^{2p s_0} 2^{-2p(s_1-s_0)}\,.
\]
By integrating in time one gets \eqref{basicStima}.
\end{proof}

The second result of this section provides a control 
of the growth of \emph{high} Sobolev norms.
\begin{theorem}{\bf (Control of high norms).}\label{thm:energyHighhigh}
Fix $s\geq s_1\geq s_0>d/2$.
Then, if $u$ is a solution of \eqref{QLNLS} satisfying  
 Hypothesis \ref{hyp:local}, one has the a priori estimates 
\begin{equation*}%\label{basicStima2}
|u(t)|_{s}^{2}\leq |u_0|_{s}^{2}
+\frac{2\mathtt{M}^{2p(s-s_1+s_0)}}{2^{2p(s_1-s_0)}}\int_{0}^{t}
|u(\tau)|^{2p}_{s_1}|u(\tau)|^{2}_{s}d\tau\,,
\end{equation*}
for every $t\in [0,T]$. As a consequence one also has
\begin{equation}\label{basicStima3}
|u(t)|^{2}_{s}\leq |u_0|^{2}_{s}\exp\Big\{ \frac{2\mathtt{M}^{2p (s-s_1+s_0)}}{2^{2p(s_1-s_0)}}\int_{0}^{t}
|u(\s)|_{s_1}^{2p}d\s\Big\}\,,
\qquad \forall\, t\in[0,T]\,. 
\end{equation}
\end{theorem}
\begin{proof}
We study the evolution of the high Sobolev norms 
$|u(t)|_{s}$ along the flow of \eqref{QLNLS}.
One has
\[
\begin{aligned}
\pa_{t}|u(t)|_{s}^{2}
&\stackrel{\eqref{normaJapJap}}{=}\pa_{t}(\jjap{D}^{s}u,\jjap{D}^{s}u)_{L^{2}}
\stackrel{\eqref{QLNLS}}{=}2{\rm Re}(\ii \jjap{D}^{s}|u|^{2p}u,\jjap{D}^{s}u)_{L^{2}}
\\&
\leq 2||u|^{2p}u|_{s}|u|_{s}
\stackrel{\eqref{confort2}}{\leq }2\mathtt{M}^{2p s}|u|_{s_0}^{2p}|u|_{s}^{2}
\stackrel{\eqref{confort3}}{\leq }
2\mathtt{M}^{2p s} (2\mathtt{M})^{-2p(s_1-s_0)}|u|_{s_1}^{2p}|u|_{s}^{2}\,.
\end{aligned}
\] 
The thesis follows by integrating in time and Lemma 
\ref{drago}-$(i)$.
\end{proof}

\subsection{Improved estimates}\label{sec:improved}

The core of this section is the following result.
\begin{theorem}{\bf (Improved a priori energy estimates).}\label{thm:energy}
Let $s_0>d/2$, $s_1\geq s_0+2$ and
consider 
$s\geq s_1+1$.
Then, if $u$ is a solution of \eqref{QLNLS}
satisfying Hypothesis \ref{hyp:local} with 
$s_0\rightsquigarrow{s}_1$,
 there exist an absolute constant $\mathtt{C},\mathtt{M}>0$ such that
\begin{equation}\label{energyapriori}
|u(t)|_{s}^{2}\leq \mathtt{C}^{2p s}\Big(|u_0|_{s}^{2}
+\frac{\mathtt{M}^{2p (s-s_1+s_0)}}{2^{2p(s_1-s_0)}}\int_{0}^{t}
|u(\tau)|^{2p+2(1-\alpha)}_{s_1}
|u(\tau)|_{s}^{2\alpha}d\tau\Big)\,,
\end{equation}
for every $t\in [0,T]$, and where
\[
\alpha:=1-\frac{1}{2(s-s_1)}\,. 
\]
As a consequence one also has
 \begin{equation}\label{buttalapasta}
|u(t)|_{s}^{2}\leq \mathtt{C}^{2ps(s-s_1)}
\Big[
|u(0)|_{s}^{2}
+
\left(
\frac{\mathtt{M}^{2p (s-s_1+s_0)}}{2^{2p(s_1-s_0)}}
\int_{0}^{t}
|u(\s)|_{s_1}^{2p+\frac{1}{s-s_1}}
d\s\right)^{2(s-s_1)}
\Big]\,,
 \end{equation}
 for any $t \in[0,T]$ and 
for some absolute constant $\mathtt{M}>0$ 
possibly larger than $\mathtt{C}$.
\end{theorem}

The proof of Theorem above is more subtle (and involves several arguments)
with respect to the energy estimates performed in subsection \ref{sec:basicbasic}.
The first result we need is the following. 

\begin{proposition}{\bf (Paralinearization of NLS).}\label{prop:NLSparalinearizzata}
We have that
the equation \eqref{QLNLS}
is equivalent to the following system (recall  \eqref{paraprod}):
\begin{equation}\label{sistemaNLS}
\pa_{t}U=-\ii E\Delta U+\ii E\mathcal{A}(u)U+\ii E\mathcal{B}(u)U+\mathcal{R}(u)U\,,
\qquad U:=\vect{u}{\bar{u}}\,,\quad  E:=\sm{1}{0}{0}{-1}\,,
\end{equation}
where 
\begin{equation}\label{simboliabbis}
\mathcal{A}(u):=\left(\begin{matrix}
T_{a} & 0 
\\0 & T_{a}
\end{matrix}\right)\,,
\quad 
\mathcal{B}(u):=
\left(\begin{matrix}
0 & T_{b} 
\\ T_{\ov{b}} & 0 
\end{matrix}\right)\,,
\qquad 
\mathcal{R}(u):=\left(\begin{matrix}
R_{1}(u) & R_{2}(u) 
\vspace{0.2em}\\\ov{R_2}(u) & \ov{R_1}(u)
\end{matrix}\right)
\end{equation}
with
\begin{equation}\label{simboliab}
a:=a(u):=(p+1)|u|^{2p}\,,\qquad 
b:=b(u):=p |u|^{2(p-1)}u^2
\end{equation}
and where the remainders 
$R_1$, $R_{2}$ are $p$-admissible remainders according to 
Definition \ref{def:pAdmissible}. In particular
for any $\rho\geq0$
there exists a constant $\mathtt{M}:=\mathtt{M}(\rho)>0$ 
such that 
for any %$s\geq s_0+\rho$ and 
$u\in B_1(H^{s_0+\rho}(\T^{d};\C))$
one has
\begin{equation*}%\label{morteneraparadiff}
|R_{j}(u)[h]|_{s+\rho}\leq \mathtt{M}^{2p\max\{s,s_0\}} |u|_{s_0+\rho}^{2p} |h|_{s}\,,
\qquad 
\forall \, h\in H^{s}(\T^{d};\C)\,,\;\; j=1,2\,, s>0\,.
\end{equation*}
\end{proposition}
\begin{proof}
It follows straightforward by \eqref{QLNLS} and Proposition \ref{prop:paralinp}.
\end{proof}

\begin{remark}\label{sapone}
Notice that, in view of \eqref{simboliab}, Hypothesis \ref{hyp:local} 
and Lemma \ref{stimaNonlin},
we have
\[
\sup_{t\in[0,T]}\big(|a|_{s_0}+|b|_{s_0}\big)\leq C^{2ps}|u|_{s_0}^{2p}\,,\qquad 
\sup_{t\in[0,T]}\big(|a|_{s}+|b|_{s}\big)\leq C^{2ps}|u|_{s_0}^{2p-1}|u|_{s}<+\infty\,,
\]
for some $C>0$ independent of $s$.
\end{remark}

Since \eqref{sistemaNLS} is a system it is convenient
to work on vector valued function spaces $H^{s}(\T^{d};\C^2)$
endowed with the standard product norm. With abuse of notation
we shall still denote it as $|\cdot|_{s}$.
Consider the subspace
\begin{equation*}%\label{subreal}
\mathcal{U}:=\{(u,v)\in L^{2}(\T^{d};\C^2)\,:\, \bar{u}=v\}\,.
\end{equation*}
Notice  that 
only have to work in the subspace 
$H^{s}(\T^{d};\C^2)\cap\mathcal{U}$.
We say that an  operator in $\mathcal{L}(H^{s}(\T^{d};\C^2))$
is \emph{real-to-real} if preserves the subspace $\mathcal{U}$.

\subsubsection{Block-diagonalization up to smoothing remainders}
In this section we block-diagonalize system 
\eqref{sistemaNLS} up to remainders which are  regularizing.
We have the following lemma.
\begin{lemma}\label{lem:stimemappa}
Consider the function 
\begin{equation}\label{generator}
c:=c(u):=\tfrac{1}{2}b(u)=\tfrac{1}{2}|u|^{2(p-1)}u^2\,,
\end{equation}
for $u=u(t,x)$ satisfying Hypothesis \ref{hyp:local} with $s_0>d/2$.
One has that  $c\in H^{s}(\T^{d};\C)$ and 
\begin{align}
|c|_{q}&\lessdot_{s}|u|^{2p-1}_{s_0}|u|_{q}\,,\qquad 
\forall\, s_0\leq q\leq s\,,\quad \forall t\in [0,T]\,.
\label{nave1}
\end{align}
If in addition $u=u(t,x)$ belongs to $H^{s}(\T^{d};\C)$ with $s\geq s_0+2$, one has 
\begin{align}
|\pa_{t}c|_{q}&\lessdot_{q}|u|^{2p-1}_{s_0+2}|u|_{q+2}\,,
\qquad \forall\, s_0+2\leq q+2\leq s\,,\quad \forall\, t\in[0,T]\,.
\label{nave2}
\end{align}
\end{lemma}
\begin{proof}
Applying  Lemma \ref{stimaNonlin} 
to \eqref{generator} one obtains \eqref{nave1}.
Let us check the bound \eqref{nave2}. First of all, recall that  
$u$ solves (recall Hypothesis \ref{hyp:local})
$\pa_{t}u=-\ii \Delta u+\ii |u|^{2p}u$, 
and hence
\begin{equation}\label{nave3}
|\pa_{t}u|_{q}\leq |\Delta u|_{q}+||u|^{2p}u|_{q}
\stackrel{\eqref{confort}}{\lessdot_{q}}
|u|_{q+2}+|u|_{s_0}^{2p}|u|_{q}
\stackrel{\eqref{hypo1}}{\lessdot_{q}}|u|_{q+2}\,,
\end{equation}
for any $t\in[0,T]$ as long as
$q+2\leq s$. 
We also note that, by \eqref{generator}, that
\[
|\pa_{t}c|_{q}\lessdot_{q}|u|_{s_0+2}^{2p}|u|_{q+2}\,,
\] 
for any $t\in[0,T]$ where we used again Lemma 
\ref{stimaNonlin} and the estimate \eqref{nave3}.
Therefore \eqref{nave2} follows.
\end{proof}

Fix $N>R$ 
and (recalling \eqref{def:truncat})  
define the operators
\begin{equation}\label{propostamappa1}
\Phi(u)[\cdot]:=\uno+\mathcal{G}^{>N}(u)\,,
\qquad
\Gamma(u)[\cdot]:=\uno-\mathcal{G}^{>N}(u) 
\end{equation}
where
\begin{equation*}
\mathcal{G}^{>N}(u):=\left(
\begin{matrix}
0 & T^{>N}_{c}\circ\langle D\rangle^{-2}
\\
T^{>N}_{\ov{c}}\circ\langle D\rangle^{-2} & 0
\end{matrix}
\right)\,.
\end{equation*}
We now show that the map $\Phi(u)$ is well-posed 
and invertible.

\begin{lemma}\label{lem:inverto}
Let $s\geq s_0>d/2$ and  $u\in H^{s}(\T^{d};\C)$
satisfying  Hypothesis \ref{hyp:local}. 
Then %for any $s\geq s_0$
the following holds:

\noindent
$(i)$ there exists an absolute constant $\mathtt{C}>0$ such that for any $s'\geq s_0$\footnote{Note that the condition $s'\geq s_0$ of items $(i)$-$(iv)$ could be relaxed up to  $s'>0$. Accordingly the constants will depend on $s_0$
as in Lemma \ref{azioneSimboo}. }
\begin{align}
|\mathcal{G}^{>N}(u)V|_{s'+2}&\leq 
\mathtt{C}^{2ps'}|u|^{2p}_{s_0}|V|_{s'}\,, \label{nave10}
\\
|\Phi(u)V|_{s'}+ |\Gamma(u)V|_{s'}&\leq
|V|_{s'}(1+ \mathtt{C}^{2ps'}|u|^{2p}_{s_0})\,, \label{nave12}
\end{align}
for any $V\in H^{s'}(\T^{d};\C^2)$.

\noindent
$(ii)$ One has 
$\Gamma(u)\circ\Phi(u)[\cdot]=\uno+\mathcal{Q}^{>N}(u)[\cdot]$
where $\mathcal{Q}^{>}(u)[\cdot]$ is a real-to-real operator
satisfying  for $s'\geq s_0$
\begin{align}
|\mathcal{Q}^{>N}(u)V|_{s'+4}&\leq \mathtt{C}^{2ps'}
|u|^{2p}_{s_0}|V|_{s'}\,,\label{nave13}
\\
|\mathcal{Q}^{>N}(u)V|_{s'+3}&\leq \mathtt{C}^{2ps'}  N^{-1}
|u|^{2p}_{s_0}|V|_{s'}\,,\label{nave14}
\end{align}
for some absolute constant $\mathtt{C}>0$
and  any $V\in H^{s'}(\T^{d};\C^2)$.

\noindent
$(iii)$
For $s'\geq s_0$, $N>\mathtt{C}^{2ps'}$, with $\mathtt{C}$ 
the constant appearing in \eqref{nave14}, 
the $\uno+\mathcal{Q}^{>N}(u)[\cdot]$ is invertible 
and 
$(\uno+\mathcal{Q}^{>N}(u)[\cdot])^{-1}=\uno+\widetilde{\mathcal{Q}}^{>N}(u)$
with
\begin{equation}\label{nave15}
|\widetilde{\mathcal{Q}}^{>N}(u)V|_{s'+3}\leq 
\mathtt{C}^{2ps'}  N^{-1}
|u|^{2p}_{s_0}|V|_{s'}\,
\end{equation}
for  any $V\in H^{s'}(\T^{d};\C^2)$. Moreover 
$\Phi^{-1}(u):=(\uno+\widetilde{Q}^{>N}(u))\circ \Gamma(u)$
satisfies 
\begin{equation}\label{nave16}
|\Phi^{-1}(u)V|_{s'}\leq
|V|_{s'}(1+ \mathtt{C}^{2ps'}|u|^{2p}_{s_0})
\end{equation}
for some absolute constant $\mathtt{C}>0$
and  any $V\in H^{s'}(\T^{d};\C^2)$.

\noindent
$(iv)$ 
if in addition $u\in H^{s}(\T^d;\C)$, $s\geq s_0+2$, then 
for any $t\in[0,T]$ and for any $s'\geq s_0$ 
one has 
$\pa_{t}\Phi(u)=\pa_{t}\mathcal{G}^{>N}(u)$ and 
\begin{equation}\label{nave20}
|\pa_{t}\mathcal{G}^{>N}(u)V|_{s'+2}\leq \mathtt{C}^{2ps'}
|u|^{2p}_{s_0+2}|V|_{s'}\,,
\end{equation}
for some absolute constant $\mathtt{C}>0$
and  any $V\in H^{s'}(\T^{d};\C^2)$.
\end{lemma}

\begin{proof}
Item $(i)$. The bound \eqref{nave10}
follows by estimate \eqref{cespuglio3} applied to the
 operator $T_{c}^{>N}\circ\langle D\rangle^{-2}$
 and using also \eqref{nave1}.
The estimate \eqref{nave12} is just a consequence
of \eqref{nave10}.

Item $(ii)$. By an explicit computation using \eqref{propostamappa1}
we get 
\[
\mathcal{Q}^{>N}(u)=\left(
\begin{matrix}
Q_1(u) & 0
\\ 0 & \ov{Q_1}(u)
\end{matrix}
\right)\,,\qquad Q_{1}(u):=-
T_{c}^{>N}\circ\langle D\rangle^{-2}\circ 
T_{\ov{c}}^{>N}\circ\langle D\rangle^{-2}\,.
\]
Then the bound \eqref{nave13} 
follows
by Lemma \ref{lem:composition} and \eqref{nave1}.
Moreover
\[
\begin{aligned}
|Q_1(u)[v]|_{s'+3}
\stackrel{\eqref{cespuglio3}, \eqref{nave1}}{\lessdot_s}
|u|^{2p}_{s_0}|T_{\bar{c}}\circ\langle D\rangle^{-2}v|_{s'+1}
\stackrel{\eqref{calma22}, \eqref{hypo1}}{\lessdot_{s'}}
|u|_{s_0}^{2p}N^{-1}|v|_{s'}\,,
\end{aligned}
\]
for any $v\in H^{s'}(\T^{d};\C)$, which is 
the \eqref{nave14}.

Item $(iii)$. The bound \eqref{nave15} follows by Neumann series,
\eqref{hypo1}, the estimate \eqref{nave14} taking $N$
large enough (with respect to $C^{2ps'}$ for some absolute constant $C>0$) to get $|\widetilde{Q}^{>N}(u)V|_{s'+3}\leq 1/2 |V|_{s'}$.
Estimate \eqref{nave16} follows by composition using 
\eqref{nave12} and \eqref{nave15}.

Item $(iv)$.
Notice that $\pa_{t}T_{c}\circ \langle D\rangle^{-2}=T_{\pa_{t}c}\circ\langle D\rangle^{-2}$.
Therefore \eqref{nave20} follows
by \eqref{calma21}
and \eqref{nave2}.
\end{proof}

As a consequence of Lemma \ref{lem:inverto} we get the following result.
\begin{corollary}\label{equinono}
Assume Hypothesis \ref{hyp:local} (with $s>d/2$) and define
\begin{equation}\label{newvariable}
W:=\Phi(u)[U]\,,\qquad W=\vect{w}{\bar{w}}\,,\quad U=\vect{u}{\bar{u}}\,,
\end{equation}
where $\Phi(u)[\cdot]$ is the map in \eqref{propostamappa1}.
Then there exists an absolute constant $\mathtt{C}>0$ 
such that, for $s\geq s_0$,
\begin{equation*}%\label{equino}
\mathtt{C}^{-ps}|u|_{s}\leq |w|_{s}\leq \mathtt{C}^{ps}|u|_{s}\,,
\qquad \forall\, t\in[0,T]\,.
\end{equation*}
%for any $t\in[0,T]$.
\end{corollary}

\begin{proof}
The thesis follows using \eqref{nave12}, \eqref{nave16} and \eqref{hyp:local}.
\end{proof}

In the following we shall study the evolution of
the \emph{modified energy} $|W|_{s}$ instead of $|U|_{s}$. 
In order to do this
we need the following preliminary result.

\begin{proposition}{\bf (Block-diagonalization).}\label{prop:block}
Let $s_0>d/2$ and
assume Hypothesis \ref{hyp:local} with $s_0\rightsquigarrow s_0+2$. 
Then the function $W$ in \eqref{newvariable} satisfies 
\begin{equation}\label{newNLS}
\pa_{t}W=-\ii E\Delta W+\ii E\mathcal{A}(u)W+\mathcal{S}(u)W\,,
\end{equation}
where $\mathcal{A}(u)$ is given 
in \eqref{simboliabbis}
and the remainder $\mathcal{S}(u)$ has the form
\begin{equation}\label{impero}
\mathcal{S}(u):=\left(\begin{matrix}
S_{1}(u) & S_{2}(u) 
\vspace{0.2em}\\\ov{S_2}(u) & \ov{S_1}(u)
\end{matrix}\right)
\end{equation}
for some operators $S_{j}(u)$, $j=1,2$, satisfying
the following. For any $s'\geq s_0$ one has
\begin{equation}\label{restofinale}
|S_{j}(u)h|_{s'+1}\lessdot_{s}|u|_{s_0+2}^{2p}|h|_{s'}\,,
\qquad \forall\, h\in H^{s'}(\T^{d};\C)\,.
\end{equation}
%for any $h\in H^{s'}(\T^{d};\C)$.
\end{proposition}

\begin{proof}
Fix $N:=\widetilde{\mathtt{C}}^{2ps}$ for some 
absolute constant
$\widetilde{\mathtt{C}}>0$ larger than $\mathtt{C}$ appearing in item 
$(iii)$ of Lemma \ref{lem:inverto}.
Recalling \eqref{sistemaNLS} and differentiating \eqref{newvariable}
we obtain
\begin{align*}
\pa_{t}W&=\Phi(u)[\pa_{t}U]+(\pa_{t}\Phi(u))[U]\nonumber
\\&
=-\ii \Phi(u)E\Delta\Phi^{-1}(u)W
+\ii \Phi(u)E\mathcal{A}\Phi^{-1}(u)W
+\ii \Phi(u)E\mathcal{B}(u)\Phi^{-1}(u)W
\\&\quad+\ii \Phi(u)\mathcal{R}(u)\Phi^{-1}(u)W
+(\pa_{t}\mathcal{G}^{>N})(u)\Phi^{-1}(u)W\,.
\end{align*}
By item $(iii)$ in Lemma \ref{lem:inverto}
we have
\begin{align}
\pa_{t}W&=-\ii \Phi(u)E\Delta \Gamma(u)W
+\ii \Phi(u)E\mathcal{A}(u)\Gamma(u)W
+ \ii \Phi(u)E\mathcal{B}(u)\Gamma(u)W
+\widetilde{\mathcal{S}}(u)W
\label{impero2}
\end{align}
where 
\begin{equation*}
\begin{aligned}
\widetilde{\mathcal{S}}(u)&:=
-\ii \Phi(u)E\Delta\widetilde{Q}^{>N}(u)\Gamma(u)
+ \ii \Phi(u)E(\mathcal{A}(u)+\mathcal{B}(u))
\widetilde{Q}^{>N}(u)\Gamma(u)
\\&\qquad
+\ii \Phi(u)\mathcal{R}(u)\Phi^{-1}(u)W
+(\pa_{t}\mathcal{G})(u)\Phi^{-1}(u)W\,.
\end{aligned}
\end{equation*}
By an explicit computation one has that $\widetilde{\mathcal{S}}$
is real-to-real, hence has the form 
\eqref{impero}.
Moreover, by Lemma \ref{lem:inverto},
Proposition \ref{prop:NLSparalinearizzata},
Lemma \ref{azioneSimboo} and recalling \eqref{hypo1}
one can check that
the remainder 
$\widetilde{\mathcal{S}}(u)$ is a one-smoothing 
operator
satisfying the bound \eqref{restofinale}.

We now study each summand appearing the 
right hand side of \eqref{impero2}.
Using \eqref{propostamappa1}, estimate \eqref{nave10}
Lemma \ref{azioneSimboo} and \eqref{confort}
it is easy to check that
\[
\ii \Phi(u)E\mathcal{A}(u)\Gamma(u)W
+ \ii \Phi(u)E\mathcal{B}(u)\Gamma(u)W
=\ii E\mathcal{A}(u)+\ii E\mathcal{B}(u)+\mathcal{S}_1(u)
\]
for some remainder $\mathcal{S}_1(u)$ of the form 
\eqref{impero} and satisfying \eqref{restofinale}.
By an explicit computation we also have
\begin{equation}\label{impero3}
-\ii \Phi(u)E\Delta \Gamma(u)=-\ii E\Delta
+\Big[\mathcal{G}^{>N}(u),-\ii E\Delta   \Big]
+\ii \mathcal{G}^{>N}(u)E\Delta \mathcal{G}^{>N}(u)\,.
\end{equation}
%Recalling that $s\geq s_0+2$, w
We note that, for any $s'\geq s_0$,
\[
\begin{aligned}
|\mathcal{G}^{>N}(u)E\Delta \mathcal{G}^{>N}(u)V|_{s'+2}
&\stackrel{\eqref{nave10}}{\lessdot_{s'}}
|u|_{s_0}^{2p}|\Delta \mathcal{G}^{>N}(u)V|_{s'}
\lessdot_{s'}
|u|_{s_0}^{2p}
| \mathcal{G}^{>N}(u)V|_{s'+2}
\stackrel{\eqref{nave10}, \eqref{hypo1}}{\lessdot_{s'}}
|u|_{s_0}^{2p}|V|_{s'}\,,
\end{aligned}
\]
for any $V\in H^{s'}(\T^{d};\C^2)$. This implies that 
\eqref{impero3}, up to remainders satisfying \eqref{restofinale},
becomes
\[
-\ii \Phi(u)E\Delta \Gamma(u)=-\ii E\Delta
+\ii E\sm{0}{G(u)}{\ov{G}(u)}{0}\,,
%+\ii E\left(
%\begin{matrix}
%0 & G(u)
%\\\ov{G}(u)& 0
%\end{matrix}
%\right)
\]
where
\begin{equation}\label{impero5}
G(u):= T_{c}^{>N}\langle D\rangle^{-2}\Delta
+ \Delta T_{c}^{>N}\langle D\rangle^{-2} \,.
\end{equation}
To summarize we have obtained
\begin{align*}
\pa_{t}W&=-\ii E\Delta W +\ii \mathcal{A}(u)W + \mathfrak{R}(u)
+\ii E\sm{0}{T_{b}\,+G(u)}{T_{\ov{b}}\,+\ov{G}(u)}{0}\,,
\end{align*}
for some remainder $\mathfrak{R}(u)$ of the form
\eqref{impero} satisfying 
\eqref{restofinale}.
The thesis follows by proving that 
the term $T_{b}+G(u) $
is actually a smoothing remainder satisfying \eqref{restofinale}.
First of all we write
\begin{equation}\label{cancello}
T_{b}+G(u)=T_{b}^{\leq N}+T_{b}^{>N}+G(u)\,,
\end{equation}
where $N\sim \mathtt{C}^{2ps'}$ has been fixed in item $(iii)$
of Lemma \ref{lem:inverto}.
Now, by Lemma \ref{danguard}, we have, for any $s'\geq s_0$,
\begin{equation*}%\label{danguard1}
|T_{b}^{\leq N}h|_{s'+1}\stackrel{\eqref{calma20}}{\lessdot_{s'}}N |u|^{2p}_{s_0} |h|_{s'}
\lessdot_{s'}|u|^{2p}_{s_0} |h|_{s'}\,,\qquad
\forall\, h\in H^{s'}(\T^{d};\C)\,.
\end{equation*}
Moreover, 
given any $h\in H^{s'}(\T^{d};\C)$
we have, for any $j\in\Z^{d}$,
\[
\begin{aligned}
\widehat{(T^{>N}_bh)}(j)+\widehat{(G(u)h)}(j)&
\stackrel{\eqref{impero5}}{=}\sum_{\eta\in\Z^d,|\eta|>N}
\chi_{\epsilon}\left(
\frac{|j-\eta|}{\langle\eta\rangle}
\right)r(j,\eta)\widehat{h}(\eta)\,,
\end{aligned}
\]
where
\[
\begin{aligned}
r(j,\eta)&:=
\widehat{b}(j-\eta) -\widehat{c}(j-\eta)\langle \eta\rangle^{-2}
\big(
|\eta|^2+|j|^2
\big)
\\&\stackrel{\eqref{generator}}{=}
\widehat{b}(j-\eta)\big[ 
1-\frac{1}{2\langle\eta\rangle^2}(|\eta|^{2}+|j|^{2})
\big]
=\widehat{b}(j-\eta)\Big[
1-\frac{|\eta|^2}{1+|\eta|^{2}}-\frac{|j|^2-|\eta|^{2}}{2(1+|\eta|^{2})}
\Big]\,.
\end{aligned}
\]
Therefore one gets
\[
| \widehat{(T^{>N}_bh)}(j)+\widehat{(G(u)h)}(j)|
\lesssim 
\sum_{\eta\in\Z^d}
|\widehat{b}(j-\eta)|\frac{|j-\eta|}{\langle\eta\rangle}
\chi_{\epsilon}\left(
\frac{|j-\eta|}{\langle\eta\rangle}
\right)\,.
\]
Reasoning as in the proof of item 
$(i)$ of Lemma \ref{azioneSimboo} and using \eqref{confort} 
we deduce
\begin{equation*}%\label{danguard2}
|(T^{>N}_b+G(u))h|_{s'+1}\lessdot_{s}|u|_{s_0+1}^{2p}|h|_{s'}\,,\qquad \forall\,
h\in H^{s'}(\T^{d};\C)\,.
\end{equation*}
By the discussion above
%\eqref{danguard1}-\eqref{danguard2} 
we deduce that
the term in \eqref{cancello} can be absorbed in the 
remainder satisfying the bound \eqref{restofinale}.
This concludes the proof.
\end{proof}

\subsubsection{Conclusions}
We are now in position to conclude the proof of the main result of this section.
\begin{proof}[{\bf Proof of Theorem \ref{thm:energy}}]
Under Hypothesis \ref{hyp:local} 
the system \eqref{sistemaNLS} is equivalent to \eqref{QLNLS}.
Moreover applying Lemma \ref{lem:inverto}
and Proposition \ref{prop:block}
we have that if $u$ solves \eqref{sistemaNLS} then
$w$ solves \eqref{scalarNLSdiag}
where the operators $\mathcal{S}_j(u)$, $j=1,2$, satisfy
the bound \eqref{restofinale}.
Recalling \eqref{simboliabbis} and \eqref{newvariable}, we have that 
\eqref{newNLS} is equivalent 
\begin{equation}\label{scalarNLSdiag}
\pa_{t}w=-\ii \Delta w+\ii T_{a}w+\mathcal{S}_1(u)[w]+\mathcal{S}_2(u)[\bar{w}]\,.
\end{equation}

We study the evolution of the Sobolev norm $|w|_{s}$
when $s\geq s_0+2$ with %$\mathfrak{s}_0:=s_0+3$, 
$s_0>d/2$.
We have
\begin{align}
\pa_{t}|w|_{s}^{2}&\stackrel{\eqref{normaJapJap}}{=}
\pa_{t}(\jjap{D}^sw,\jjap{D}^{s}w)_{L^2}
=2{\rm Re}(\jjap{D}^{s}\pa_{t}w,\jjap{D}^{s}w)_{L^2}
\nonumber 
\\&
\stackrel{\eqref{scalarNLSdiag}}{=}
2{\rm Re}(\jjap{D}^{s}\ii T_{a}w,\jjap{D}^{s}w)_{L^2}
\label{impero10}
\\&
\qquad+2{\rm Re}(\jjap{D}^{s}(\mathcal{S}_1(u)[w]
+\mathcal{S}_{2}(u)[\bar{w}]), \jjap{D}^{s}w)_{L^2}
\label{impero11}
\end{align}
First of all, by Cauchy-Schwarz inequality and 
since $s\geq s_0+2$ (implying  $s-1\geq s_0$), we have
\begin{equation}\label{impero13}
|\eqref{impero11}|\lesssim
|w|_{s}\sup_{j=1,2}|\mathcal{S}_j(u)[w]|_{s}
\stackrel{\eqref{restofinale}}{\lessdot_s}
|u|^{2p}_{s_0+2}|w|_{s-1}|w|_{s}\,.
\end{equation}
%using that $s-1\geq s_0$ since we even have $s\geq s_0+2$.
Moreover
\begin{equation}\label{impero20}
\begin{aligned}
|\eqref{impero10}|
&=|( \ii\big[\jjap{D}^{2s},T_{a}\big]w , w)_{L^2}|
= |( \ii\jjap{D}^{-s+1}\big[\jjap{D}^{2s},T_{a}\big]w , 
\jjap{D}^{s-1}w)_{L^2}|
\\&\leq
|w|_{s-1}| \big[\jjap{D}^{2s},T_{a}\big]w |_{-s+1}
\stackrel{\eqref{cespuglio4}}{\lessdot_{s}}
|w|_{s-1}|a|_{s_0+1}|w|_{s}
\lessdot_{s}
|u|_{s_0+1}^{2p}|w|_{s-1}|w|_{s}\,,
\end{aligned}
\end{equation}
where we used 
\eqref{cespuglio4}
with $p\rightsquigarrow 2s$, \eqref{simboliab} and Remark \ref{sapone}. 
Collecting together \eqref{impero13} and \eqref{impero20}
and integrating in time we get
\[
|w(t)|_{s}^{2}\leq |w(0)|_{s}^{2}+\mathtt{M}^{2ps}\int_{0}^{t}
|u(\tau)|_{s_0+2}^{2p}|w(\tau)|_{s-1}|w(\tau)|_{s}d\tau
\]
for any $t\in [0,T]$ and some absolute constant $\mathtt{M}>0$.
The latter bound,
together with the equivalence
in Corollary
\ref{equinono}, implies 
\begin{equation}\label{energyaprioriBIS}
|u(t)|_{s}^{2}\leq \mathtt{C}_1^{2ps} \Big(|u(0)|_{s}^{2}+\mathtt{M}^{2ps}\int_{0}^{t}
|u(\tau)|_{s_0+2}^{2p}|u(\tau)|_{s-1}|u(\tau)|_{s}d\tau\Big)\,,
\end{equation}
for some absolute constant $\mathtt{C}_1>0$.
Notice  that (by assumption)
$s_0+2\leq s_1\leq s-1$ %\leq s$ (recall that $s\geq s_0+3$) 
and hence
\[
s-1=(1-\lambda)s_1+\lambda s\,\quad \Leftrightarrow\quad  \lambda=1-\frac{1}{s-s_1}\,.
\]
Then, by Lemma \ref{interopolo}
we get
\[
|u|_{s-1}|u|_{s}\leq |u|_{s_1}^{1-\lambda}|u|_{s}^{1+\lambda}
\leq |u|_{s_1}^{\frac{1}{s-s_1}}|u|_{s}^{2\alpha}\,,
\quad \alpha:=1-\frac{1}{2(s-s_1)}\,,
\] 
for any $ t\in[0,T]$. 
Then by \eqref{energyaprioriBIS}, \eqref{scaling property} (with $R\rightsquigarrow 2\mathtt{M}$, 
$s_0\rightsquigarrow s_0+2$, $s\rightsquigarrow s_1$)
we get
\[
|u(t)|_{s}^{2}\leq \mathtt{C}_1^{2p s}\Big(|u_0|_{s}^{2}
+\frac{\mathtt{M}^{2p (s-s_1+s_0)} (2\mathtt{M})^{4p}}{2^{2p(s_1-s_0)}}\int_{0}^{t}
|u(\tau)|^{2p+2(1-\alpha)}_{s_1}
|u(\tau)|_{s}^{2\alpha}d\tau\Big)
\]
The latter bound 
implies \eqref{energyapriori} choosing $\mathtt{C}>0$ large with $\mathtt{M}$.

In order to prove \eqref{buttalapasta} we reason as follows.
By estimate \eqref{energyapriori}
and Lemma \ref{drago}-$(ii)$ one deduce
\begin{equation*}
(|u(t)|_{s}^{2})^{1-\alpha}\leq 
(\mathtt{C}_1^{2ps}|u(0)|_{s}^{2})^{1-\alpha}
+(1-\alpha)
\frac{\mathtt{C}_1^{2ps}\mathtt{M}^{2p (s-s_1+s_0)}}{2^{2p(s_1-s_0)}}
\int_{0}^{t}
|u(\s)|_{s_1}^{2p+\frac{1}{s-s_1}}
d\s\,,
\qquad \forall \, t\in[0,T]\,.
\end{equation*}

By a simple computation (see footnote \eqref{ciaoneproprio})
%\footnote{
%Using  that $(x+y)^{q}\leq 2^{q-1}(x^{q}+y^{q})$,
%$x,y\geq0$ and $q>1$\,.}
, recalling that $1-\alpha=\tfrac{1}{2(s-s_1)}$, 
we deduce
\[
|u(t)|_{s}^{2}\leq \mathtt{C}^{2ps(s-s_1)}
\Big[
|u(0)|_{s}^{2}
+
\left(
\frac{\mathtt{M}^{2p (s-s_1+s_0)}}{2^{2p(s_1-s_0)}}
\int_{0}^{t}
|u(\s)|_{s_1}^{2p+\frac{1}{s-s_1}}
d\s\right)^{2(s-s_1)}
\Big]\,,
\]
which implies 
\eqref{buttalapasta} for some suitable $\mathtt{C }>0$ large enough.
\end{proof}

\section{Conclusions}\label{sec:basicestimate}
In this section we give the proof of the main results Theorems \ref{thm:main1}-\ref{thm:main3}.

\subsection{Proof of Theorem \ref{thm:main1}}\label{proof:thm1}
In this section we prove the our main  Theorem \ref{thm:main1}.
First of all, consider the constant $\mathtt{M}>0$ 
provided by Lemma \ref{lem:confort2}.
Let us fix $R:=2\mathtt{M}$ and consider (see \eqref{piccolezza dati})
any initial condition $u_0\in H^{s_1}(\T^{d};\C)$ 
satisfying
\begin{equation}\label{piccolezza dati2}
\|u\|_{H^{s_1}}:=\|u\|_{L^{2}}+\|\pa_{x}^{s_1}u\|_{L^{2}}\leq \e\,,
\qquad 
R^{s_1}\|u\|_{L^{2}}\leq \e\,,
\end{equation}
and let
us define 
\begin{equation}\label{deltadeserto}
\delta:=\|u_0\|_{H^{s_1}}+(2\mathtt{M})^{s_1}\|u_0\|_{L^{2}}\,.
\end{equation}
Recalling \eqref{normaJapJap}-\eqref{def:japjapModificato}
we have that condition \eqref{piccolezza dati2} and \eqref{equiIncredibile1} imply
$|u_0|_{s_1}\leq \delta \leq 2\e$\,.

\smallskip
\noindent
{\bf Proof of \eqref{storm}-\eqref{stimaccia}.}
By classical local existence theory we have that there exists 
a unique solution in
\[
C^{0}([0,T_{loc}];H^{s_1}(\T^{d};\C))\cap C^{1}([0,T_{loc}];H^{s_1-2}(\T^{d};\C))\,,
\]
for some $T_{loc}>0$ possibly small, i.e. Hypothesis 
\ref{hyp:local} is verified.
We now show that actually the solution exists 
over a time interval $[0,T]$ with $T\geq T_{loc}$ 
satisfying the bound
in \eqref{storm}.

Let $B^{s_{1}}(2\delta)$ be the closed ball of radius $2\delta$, centered at the origin in $H^{s_{1}}(\T^{d};\C)$. 
And, considering the solution $u(t)$ with initial condition $u^{0}\in B^{{s_{1}}}({\delta})$ 
introduced just above, let us define the time of escape from the ball as 
\begin{equation*}%\label{time escape}
\tau_{e} : = \inf\set{t\in\R\, : u(t)\not\subset B^{s_{1}}(2\delta)}\,.
\end{equation*}
Notice that a priori one has 
$\sup_{t\in[0,\tau_e)} |u(t)|_{s_1}\leq 2\delta$.
If the set is empty then we have ``perpetual'' stability and the solution stays in the ball 
of radius $2\delta$ as long as the solution exists. 
If the set is not empty then
 we shall show that 
  \begin{equation}\label{ginogino2}
\tau_{e} \geq  \frac{1}{8^{2p+2}}\mathtt{M}^{-2ps_0} \e^{-2p}2^{2p(s_1-s_0)}:=T_{good}\,,
\end{equation} 
where we have assumed, without loss of generality, that $\tau_e>0$.
We prove the claim by contradiction assuming $\tau_{e}<T_{good}$.

Theorem \ref{thm:energyBasic} applies with $T\rightsquigarrow \tau_e$, 
hence 
by estimate \eqref{basicStima} and recalling $\delta\leq 2\e$
we deduce
\[
\begin{aligned}
(2\delta)^2\leq |u(\tau_e)|_{s_1}^{2}&\leq 
|u_{0}|_{s_1}^{2}
+\mathtt{M}^{2p s_0} 2^{-2p(s_1-s_0)}\int_{0}^{\tau_{e}}
|u(\s)|_{s_1}^{2p+2}d\s\,,
\\&\leq 
\delta^2+
\tau_e\mathtt{M}^{2p s_0} 2^{-2p(s_1-s_0)}4^{2p+2}\delta^{2p+2}
\\&
\leq \delta^{2}\big(1+ \frac{\delta^{2p}}{\e^{2p}}\frac{4^{2p+2}}{8^{2p+2}} \big)
\leq \delta^{2}(1+\tfrac{1}{16})
<(2\delta)^2\,,
\end{aligned}
\]
which is a contradiction. So one must have $\tau_{e}\geq T_{good}$.
Therefore
the solution $u(t)$ can be extended over a time interval $[0,T]$ with $T_{good}\leq T<\tau_e$ 
 consistently with \eqref{storm},
and that 
\begin{equation}\label{stimateo11}
\sup_{t\in[0,T]}|u(t)|_{s_1}\leq 2\delta\,.
\end{equation}
The latter bound, together with \eqref{equiIncredibile1}, implies 
\[
\|u(t)\|_{H^{s_1}}+(2\mathtt{M})^{s_1}\|u(t)\|_{L^{2}}\leq 6\delta\,,
\]
uniformly in $t\in[0,T]$, from which we deduce \eqref{stimaccia}.

\smallskip
\noindent
{\bf Proof of \eqref{stimaccia2bis}-\eqref{stimaccia2}.}
Reasoning as above we now assume that the initial condition $u_0$
satisfies
\[
|u_0|_{s_1}\leq 
\delta:=\|u_0\|_{H^{s_1}}+(2\mathtt{M})^{s_1}\|u_0\|_{L^{2}}\leq 
2\e\,,\qquad  |u_0|_{s}<+\infty\,.
\]
By classical local existence theory we have that there exists 
a unique solution 
in the space 
\[
C^{0}([0,T_{loc}];H^{s}(\T^{d};\C))\cap C^{1}([0,T_{loc}];H^{s-2}(\T^{d};\C))\,,
\]
for some $T_{loc}>0$ possibly small. 
Moreover, in view of  the smallness condition on $|u_0|_{s_1}$, 
estimate \eqref{stimateo11}
guarantees that
the \emph{low} norm $|\cdot|_{s_1}$
stay small for very long time (recall \eqref{ginogino2}):
\begin{equation}\label{staysmall}
\sup_{t\in[0,T]}|u(t)|_{s_1}\leq 2\delta\,,
\qquad 
T\geq T_{good}\,.
%:=\frac{1}{8^{2p+2}}\mathtt{M}^{-2ps_0}
%\e^{-2p} 2^{2p(s_1-s_0)}\,.
\end{equation}
We now study the evolution of the high norm
$|u(t)|_{s}$ along the flow of \eqref{QLNLS}.
By the a priori estimate \eqref{basicStima3} of Theorem \ref{thm:energyHighhigh},
together with Remark \ref{scalaBan} and \eqref{staysmall},
we also have the following  control of the \emph{high} norm: 
\begin{equation}\label{redifrancia}
|u(t)|^{2}_{s}\leq
|u_0|^{2}_{s}\exp\big\{ \frac{\mathtt{M}^{2p (s-s_1+s_0)}}{2^{2p(s_1-s_0)}} 4^{2p+1}\e^{2p}{t}\big\}\,,
\qquad \forall\, t\in[0,T]\,.
\end{equation}
By classical prolongation argument one can extend the solution on the time interval
$[0,T]$ obtaining the thesis.
The  estimate \eqref{stimaccia2bis} follows 
by \eqref{redifrancia} and \eqref{equiIncredibile1}.
The bound \eqref{stimaccia2} 
follows immediately using \eqref{staysmall}.

\subsection{Proof of Theorem \ref{thm:main3}}
We are in position to prove our main result.
Let $s\geq s_1+1$ with $s_1\geq s_0+2$ where  $s_0>d/2$.
Under the assumptions of Theorem \ref{thm:main3}
we have that Theorem \ref{thm:main1} guarantees that
(recall also the equivalence of the norms \eqref{equiIncredibile1})
\begin{equation}\label{disastro}
\sup_{t\in[0,T]}|u(t)|_{s_0+2}\leq \sup_{t\in[0,T]}|u(t)|_{s_1}\leq 2\delta\leq 4\e\,,
\end{equation}
where  $\delta $ is in \eqref{deltadeserto} and for $T\geq T_{good}$ in \eqref{ginogino2}.
%\[
%\delta:=\|u_0\|_{H^{s_1}}+(2\mathtt{M})^{s_1}\|u_0\|_{L^2}\,,
%\]
%for (recall \eqref{storm})
%\begin{equation*}%\label{disastro2}
%T\leq T_{good}:= \frac{1}{8^{2p+2}}\mathtt{M}^{-2ps_0} \e^{-2p}2^{2p(s_1-s_0)}\,.
%\end{equation*}
In particular for $4\e\leq 1$ one has that 
Hypothesis \ref{hyp:local} holds, since the 
\emph{low} norm $|u(t)|_{s_0+2}$ is controlled $[0,T]$.
%for  ${T}\leq T_{good}$.
We shall use a boot strap argument to show that actually one can 
extends the solution in the space $H^{s}(\T^{d};\C)$ 
over the time interval $[0,T]$.
For any $0\leq t\leq \widehat{T}\leq T$,
%$0\leq t\leq \widehat{T}\leq T_{good}$,
by applying Theorem \ref{thm:energy} %and Corollary \ref{coroFondamentale},
(see \eqref{buttalapasta})
 we have the a priori
bound on the $H^{s}$-norm
%. By the scaling property in Remark \ref{scaling property} we get
%(recall $\mathfrak{s}_0:=s_0+2$)
 \begin{equation*}%\label{buttalapasta2}
 \begin{aligned}
|u(t)|_{s}^{2}
 &\leq 
 \mathtt{C}^{2ps(s-s_1)}\Big[|u(0)|_{s}^{2}+  
 \Big(
 \frac{t\mathtt{M}^{2p (s-s_1+s_0)}}{2^{2p(s_1-s_0)}}
 \sup_{t\in[0,T]}|u(t)|_{s_1}^{2p}
 \Big)^{2(s-s_1)}
  \big(
 \sup_{t\in[0,T]}|u(t)|_{s_1}
 \big)^{2} \Big]
 \\&
 \stackrel{\eqref{disastro}}{\leq }
  \mathtt{C}^{2ps(s-s_1)}\Big[|u(0)|_{s}^{2}+  
 \Big(
 \frac{t\e^{2p}\mathtt{M}^{2p (s-s_1+s_0)}}{2^{2p(s_1-s_0)}}
4^{2p}
 \Big)^{2(s-s_1)}
 4\delta^2
 \Big]\,.
 %\\&\stackrel{\eqref{confort3}}{\leq}
%   \mathtt{M}^{4ps(s-s_1)}|u(0)|_{s}^{2}
%   %\\&\qquad
%   + 
%   \left[\frac{ \mathtt{M}^{2ps}    T }{ (2\mathtt{M})^{2p(s_1-s_0-2)}}
%   ( \sup_{t\in[0,T_{good}]}|u(t)|_{s_1})^{2p} 
%   \right]^{2(s-s_1)}
%   \!\!\!( \sup_{t\in[0,T_{good}]}|u(t)|_{s_1})^2
% \\&\stackrel{\eqref{disastro}}{\leq }
%  \mathtt{M}^{4ps(s-s_1)}|u(0)|_{s}^{2}+\textcolor{red}{16\e^{2}}
%  \left[
% \frac{ 2^{4p} 4^{2p}}{2^{2p(s_1-s_0)}} \e^{2p}T \frac{\mathtt{M}^{2ps}}{
%  \mathtt{M}^{2p(s_1-s_0-2)}}
%  \right]^{2(s-s_1)}\,.
%  \\&\stackrel{\eqref{disastro2}}{\leq}
%    \mathtt{M}^{4ps(s-\mathfrak{s}_0)}|u(0)|_{s}^{2}+\mathtt{M}^{4p(s-\mathfrak{s}_0)(s-s_0)}
%  2^{4p(s_1-s_0)(s-\mathfrak{s}_0)}
%    (4)^{4p(s-\mathfrak{s}_0)} 8^{-(2p+2)2(s-\mathfrak{s}_0)}
%  \,.
 \end{aligned}
 \end{equation*}
Using \eqref{equiIncredibile1} and taking the 
 square root\footnote{We used that $\sqrt{a^2+b^2}\leq a+b$ for $a,b\geq 0$.}
we get
 \[
 \begin{aligned}
\|u(t)\|_{H^s}
  &\leq 
  3 \mathtt{C}^{ps(s-s_1)}\Big(\|u(0)\|_{H^s}+ (2\mathtt{M})^{s}\|u_0\|_{L^{2}}\Big)
  \\
  &+  2 
  \mathtt{C}^{ps(s-s_1)}\Big(\|u(0)\|_{H^{s_1}}+ (2\mathtt{M})^{s_1}\|u_0\|_{L^{2}}\Big)
 \Big(
 \frac{t\e^{2p}\mathtt{M}^{2p (s-s_1+s_0)}}{2^{2p(s_1-s_0)}}
4^{2p}
 \Big)^{s-s_1}\,,
 \end{aligned}
 \]
 for $t\in[0,T]$.
 The latter bound implies \eqref{stimaIncredibleBis} and
 the thesis follows by a standard bootstrap argument. 
The bound \eqref{stimaIncredible} follows by \eqref{stimaIncredibleBis},
where $T_{good}$ is defined in 
 \eqref{patata1}.

\vspace{0.6em}

\gr{Declarations}. Data sharing not applicable to this article as no datasets 
were generated or analyzed during the current study.

\noindent
Conflicts of interest: The authors have no conflicts of interest to declare.

\vspace{0.6em}
\textbf{Acknowledgments.} 
The authors thanks Emanuele Haus and Michela Procesi for 
fruitful comments and discussions.
The authors have been  supported by the  research project 
PRIN 2020XBFL ``Hamiltonian and dispersive PDEs" of the 
Italian Ministry of Education and Research (MIUR). The authors also acknowledge the support of the INdAM-GNAMPA research project 
``Chaotic and unstable behaviors of infinite-dimensional dynamical systems" CUP\_\,E55F22000270001.

\appendix
\section{On some Gr\"onwall type inequalities}
We collect here some classical results about Gr\"owall type inequalities.
%We refer for instance to \cite{MPF1991}.
%
%\begin{lemma}\label{gronclassical}
%Consider   $a,b\in \R$ with $a<b$,  continuous
%functions $f:\R\to\R_{+}:=[0,+\infty)$ and assume that
%$x : [a,b]\to\R_{+}$  is a differentiable function
%satisfying
%\begin{equation*}%\label{drago1}
%x(t)\leq  M+  \int_{a}^{t}f(\s)x(\s) d\s\,,\quad \forall \, t\in[a,b]\,,
%\end{equation*}
%where $M\geq0$.
%Then one has
%\begin{equation*}%\label{drago2}
%x(t)\leq M\exp\big\{\int_{a}^{t}f(\s)d\s\big\}\,,\quad \forall\, t\in[a,b]\,.
%\end{equation*}
%\end{lemma}

\begin{lemma}\label{drago}
Consider  parameters $M\geq0$, $\alpha\in(0,1)$, $a,b\in \R$ with $a<b$.
Let 
%a continuous
%function 
$f:\R\to\R_{+}:=[0,+\infty)$ and 
$x : [a,b]\to\R_{+}$  be continuous functions.

\noindent
$(i)$ If the function $x(t)$ satisfies 
\begin{equation*}
x(t)\leq  M+  \int_{a}^{t}f(\s)x(\s) d\s\,,\quad \forall \, t\in[a,b]\,,
\end{equation*}
then one has
\[
 x(t)\leq M\exp\big\{\int_{a}^{t}f(\s)d\s\big\}\,, \quad \forall \, t\in[a,b]\,.
\]

\noindent
$(ii)$ If the function $x(t)$ satisfies 
\begin{equation*}%\label{drago1}
x(t)\leq  M+  \frac{1}{1-\alpha}\int_{a}^{t}f(\s)(x(\s))^{\alpha}d\s\,,\quad \forall \, t\in[a,b]\,,
\end{equation*}
then one has
\begin{equation*}%\label{drago2}
\big(x(t)\big)^{1-\alpha}\leq M^{1-\alpha}+\int_{a}^{t}f(\s)d\s\,,\quad \forall\, t\in[a,b]\,.
\end{equation*}
\end{lemma}
\begin{proof}
%%The result is classical, we refer for instance to \cite{drago2003} and \cite{MPF1991}.
%$(i)$ The result is classical, we refer for instance to \cite{MPF1991}.
%%The proof is classical and we omit it.
%
%\noindent
%$(ii)$
The result of item $(i)$ is classical and 
we refer, for instance, to \cite{MPF1991} for its proof. 
For completeness we give the proof of item $(ii)$.
We set
\[
y(t):=\frac{1}{1-\alpha}\int_{a}^{t}f(\s)(x(\s))^{\alpha}d\s\,,\quad t\in[a,b]\,,
\quad \Rightarrow \quad y(a)=0\,.
\]
We note  that $x(t)\leq (M+y(t))$ and that
\[
\pa_{t}y(t)= \frac{1}{1-\alpha}f(t)(x(t))^{\alpha}\,.
\]
We then deduce
\[
\frac{d}{dt}(M+y(t))^{1-\alpha}=(1-\alpha)\frac{\dot{y}(t)}{(M+y(t))^{\alpha}}
=
\frac{f(t)(x(t))^{\alpha}}{(M+y(t))^{\alpha}}\leq f(t)\,.
\]
Integrating from $a$ to $t$ we get
\[
(M+y(\s))^{1-\alpha}\Big|_{a}^{t}
\leq \int_{a}^{t}f(\s)d\s\,\qquad \Rightarrow\qquad
(M+y(t))^{1-\alpha}\leq M^{1-\alpha}+ \int_{a}^{t}f(\s)d\s\,.
\]
Therefore the thesis follows. 
%$(ii)$
%We set
%\[
%y(t):=\frac{1}{1-\alpha}\int_{a}^{t}f(\s)(x(\s))^{\alpha}d\s\,,\quad t\in[a,b]\,,
%\quad \Rightarrow \quad y(a)=0\,.
%\]
%We note  that $x(t)\leq (M+y(t))$ and that, by assumption, 
%%\eqref{drago1},
%\[
%\pa_{t}y(t)\leq \frac{1}{1-\alpha}f(t)\Big(M+y(t)  \Big)^{\alpha}\,,\quad t\in[a,b]\,.
%\]
%We then deduce
%%Integrating in $t$ we get
%\[
%(M+y(\s))^{1-\alpha}\Big|_{a}^{t}
%=
%\int_{a}^{t}\frac{(1-\alpha)\pa_{\s}y(\s)}{(M+y(\s))^{\alpha}}d\s
%\leq \int_{a}^{t}f(\s)d\s\,,
%\]
%which implies
%\[
%(M+y(t))^{1-\alpha}\leq M^{1-\alpha}+ \int_{a}^{t}f(\s)d\s\,.
%\]
%Therefore the thesis follows. 
%%\red{Mi sembra che la stima sia con $M$ al posto di $x(a)$}
%%from which we deduce \eqref{drago2}.
\end{proof}

%\bibliographystyle{plain}
%\bibliography{biblioRobJessNLSlong}

\end{document}